\documentclass[11pt]{article}%
\usepackage{amsmath}
\usepackage{amsfonts}
\usepackage{amssymb}
\usepackage{graphicx}%
\numberwithin{equation}{section}
\newtheorem{theorem}{Theorem}[section]

\newtheorem{corollary}[theorem]{Corollary}

\newtheorem{definition}[theorem]{Definition}

\newtheorem{lemma}[theorem]{Lemma}

\newtheorem{proposition}[theorem]{Proposition}

\newenvironment{proof}[1][Proof]{\noindent\textbf{#1.} }
{\ \rule{0.5em}{0.5em}}
\begin{document}

\title{Dunkl Operators and Related Special Functions}
\author{Charles F. Dunkl\\Department of Mathematics, University of Virginia,\\PO Box 400137, Charlottesville, VA 22904-4137}
\date{10 October 2012}
\maketitle

\section{Introduction}

This is a preliminary version of a chapter for the volume titled
\textquotedblleft Multivariable Special Functions,\textquotedblright\ edited
by T. Koornwinder, in the Askey-Bateman project.

Functions like the exponential, Chebyshev polynomials, and monomial symmetric
polynomials are pre-eminent among all special functions. They have simple
definitions and can be expressed using easily specified integers like $n!$.
Families of functions like Gegenbauer, Jacobi and Jack symmetric polynomials
and Bessel functions are labeled by parameters. These could be unspecified
transcendental numbers or drawn from large sets of real numbers, for example
the complement of $\left\{  -\frac{1}{2},-\frac{3}{2},-\frac{5}{2}%
,\ldots\right\}  $. One aim of this chapter is to provide a harmonic analysis
setting in which parameters play a natural role. The basic objects are finite
reflection (Coxeter) groups and algebras of operators on polynomials which
generalize the algebra of partial differential operators. These algebras have
as many parameters as the number of conjugacy classes of reflections in the
associated groups. The first-order operators have acquired the name
\textquotedblleft Dunkl operators.\textquotedblright

Coxeter groups are related to root systems. This chapter begins with a
presentation of these systems, the definition of a Coxeter group and the
classification of the indecomposable systems. Then the theory of the operators
is developed in detail with an emphasis on inner products and their relation
to Macdonald-Opdam integrals, a generalization of the Fourier transform, the
Laplacian operator and harmonic polynomials, and applications to particular
groups to study Gegenbauer, Jacobi and nonsymmetric Jack polynomials.

For $x,y\in\mathbb{R}^{d}$ the inner product is $\left\langle x,y\right\rangle
=\sum_{j=1}^{d}x_{j}y_{j}$, and the norm is $\left\vert x\right\vert
=\left\langle x,x\right\rangle ^{1/2}.$ A matrix $w=\left(  w_{ij}\right)
_{i,j=1}^{d}$ is called orthogonal if $ww^{T}=I_{d}$. The group of orthogonal
$d\times d$ matrices is denoted by $O\left(  d\right)  $. The standard unit
basis vectors of $\mathbb{R}^{d}$ are denoted by $\varepsilon_{i},1\leq i\leq
d$. The nonnegative integers $\left\{  0,1,2,\ldots\right\}  $ are denoted by
$\mathbb{N}_{0}$, and the cardinality of a set $S$ is denoted by $\#S$.

\begin{definition}
For $a\in\mathbb{R}^{d}\backslash\left\{  0\right\}  $, the reflection along
$a$, denoted by $s_{a}$, is defined by
\[
s_{a}x:=x-2\frac{\left\langle x,a\right\rangle }{\left\vert a\right\vert ^{2}%
}a.
\]

\end{definition}

Writing $s_{a}=I_{d}-2\left(  a^{T}a\right)  ^{-1}aa^{T}$ shows that
$s_{a}=s_{a}^{T}$ and $s_{a}^{2}=I_{d}$, that is, $s_{a}\in O\left(  d\right)
$. The matrix entries of $s_{a}$ are $\left(  s_{a}\right)  _{ij}=\delta
_{ij}-2a_{i}a_{j}/\left\vert a\right\vert ^{2}$. The hyperplane $a^{\bot
}=\left\{  x:\left\langle x,a\right\rangle =0\right\}  $ is the invariant set
for $s_{a}$, Also $s_{a}a=-a$, \ and any nonzero multiple of $a$ determines
the same reflection. The reflection $s_{a}$ has one eigenvector for the
eigenvalue $-1$, and $d-1$ independent eigenvectors for the eigenvalue $+1$
and $\det s_{a}=-1$.

For $a,b\in\mathbb{R}^{d}\backslash\left\{  0\right\}  $, let $\cos
\measuredangle\left(  a,b\right)  :=\left\langle a,b\right\rangle /\left(
\left\vert a\right\vert ~\left\vert b\right\vert \right)  $; if $a,b$ are
linearly independent in $\mathbb{R}^{d}$ then $s_{a}s_{b}$ is a plane rotation
in $\mathrm{span}_{\mathbb{R}}\left\{  a,b\right\}  $ through an angle
$2\measuredangle\left(  a,b\right)  $. Consequently for given $m=1,2,3,\ldots
,\left(  s_{a}s_{b}\right)  ^{m}=I_{d}$ if and only if $\cos\measuredangle
\left(  a,b\right)  =\cos\left(  \frac{\pi j}{m}\right)  $ for some integer
$j$. Two reflections $s_{a}$ and $s_{b}$ commute if and only if $\left\langle
a,b\right\rangle =0$, since $\left(  s_{a}s_{b}\right)  ^{-1}=s_{b}s_{a}$, in
general$.$ The conjugate of a reflection is again a reflection: suppose $w\in
O\left(  d\right)  $ and $a$ $\in\mathbb{R}^{d}\backslash\left\{  0\right\}  $
then $ws_{a}w^{-1}=s_{wa}.$

\section{Root Systems}

\begin{definition}
A \emph{root system }is a finite set $R\subset\mathbb{R}^{d}\backslash\{0\}$
such that $a,b\in R$ implies $s_{b}a\in R$. If, additionally, $\mathbb{R}a\cap
R=\{\pm a\}$ for each $a\in R$ then $R$ is said to be \emph{reduced}. The
\emph{rank} of $R$ is defined to be $\dim(\mathrm{span}_{\mathbb{R}}(R))$.
\end{definition}

Note that $a\in R$ implies $-a=s_{a}a\in R$, for any root system. By choosing
some $a_{0}\in\mathbb{R}^{d}$ such that $\langle a,a_{0}\rangle\neq0$ for all
$a\in R$ one defines the \emph{set of positive roots} to be $R_{+}=\{a\in
R:\langle a,a_{0}\rangle>0\}$. Also set $R_{-}=-R_{+}$ so that $R=R_{+}\cup
R_{-}$, the disjoint union of the sets of positive and negative roots.

\begin{definition}
The \emph{Coxeter}, or \emph{finite reflection}, \emph{group} $W(R)$ is
defined to be the subgroup of $O(d)$ generated by $\{s_{a}:a\in R_{+}\}$.
\end{definition}

The group $W(R)$ is finite because each $w\in W(R)$ fixes the orthogonal
complement of $\mathrm{span}_{\mathbb{R}}(R)$ pointwise, and for each $a\in R$
the orbit $W(R)a\subset R$, hence is finite.

There is a distinguished set $S=\{r_{1},\ldots,r_{n}\}$, $n=\dim
(\mathrm{span}_{\mathbb{R}}(R))$, of positive roots, called the \emph{simple
roots}, such that $S$ is a basis for $\mathrm{span}_{\mathbb{R}}(R)$ and $a\in
R_{+}$ implies $a=\sum_{i=1}^{n}c_{i}r_{i}$ with $c_{i}\geq0$ (see
\cite[Theorem 1.3]{Hum1}). The corresponding reflections $s_{i}=s_{r_{i}}$
($i=1,\ldots,n$) are called \emph{simple reflections}.

\begin{definition}
The \emph{length} of $w\in W(R)$ is $\ell(w):=\#(wR_{+}\cap R_{-})$.
\end{definition}

Equivalently (see \cite[Corollary 1.7]{Hum1}), $\ell(w)$ equals the number of
factors in the shortest product $w=s_{i_{1}}s_{i_{2}}\ldots s_{i_{m}}$ for
expressing $w$ in terms of simple reflections. In particular, $w$ has length
one if and only if it is a simple reflection.

For purpose of studying the group $W(R)$ one can replace $R$ by a reduced root
system (example: $\left\{  \sqrt{2}\,a/\vert a\vert: a\in R\right\}  $ is a
commonly used normalization).

\begin{definition}
The \emph{discriminant}, or \emph{alternating polynomial}, of the reduced root
system $R$ is
\[
a_{R}(x):=\prod\nolimits_{a\in R_{+}}\langle x,a\rangle.
\]

\end{definition}

If $R$ is reduced and $b\in R_{+}$ then $\left\langle s_{b}x,b\right\rangle
=-\left\langle x,b\right\rangle $ and $\#\left\{  a\in R_{+}\backslash
\{b\}:s_{b}a\in-R_{+}\right\}  $ is even. It follows that $a_{R}%
(s_{b}x)=-a_{R}(x)$, and furthermore that $a_{R}(wx)=\det(w)a_{R}(x)$, $w\in
W(R)$.

The set of reflections in $W(R) $ is exactly $\{s_{a}:a\in R_{+}\}$. This is a
consequence of a divisibility property: if $b\in\mathbb{R}^{d}\backslash\{
0\}$ and $p$ is a polynomial such that $p(s_{b}x)=-p(x)$ for all
$x\in\mathbb{R}^{d}$, then $p(x)$ is divisible by $\langle x,b\rangle$
(without loss of generality assume that $b=\varepsilon_{1}$; let
$p(x)=\sum_{j=0}^{n}x_{1}^{j} p_{j}(x_{2},\ldots,x_{d})$, then $p(s_{b}%
x)=\sum_{j=0}^{n}(-x_{1})^{j}p_{j}(x_{2},\ldots,x_{d})$ and $p(s_{b}x)=-p(x)$
implies $p_{j}=0$ unless $j$ is odd). If $w=s_{b}\in W(R)$ is a reflection
then $\det w=-1$ and so $a_{R}(wx)=-a_{R}(x)$ and hence $a_{R}(x)$ is
divisible by $\langle x,b\rangle$. Linear factors are irreducible and the
unique factorization theorem shows that some multiple of $b$ is an element of
$R_{+}$.

\begin{definition}
\label{c4cryst} The root system $R$ is called \emph{crystallographic} if
$2\langle a,b\rangle/|b|^{2}\in\mathbb{Z}$ for all $a,b\in R$.
\end{definition}

In the crystallographic case $R$ is in the $\mathbb{Z}$-lattice generated by
the simple roots and $W(R)$ acts on this lattice (see \cite[\S 10.1]{Hum2}).

If $R$ can be expressed as a disjoint union of non-empty sets $R_{1}\cup
R_{2}$ with $\langle a,b\rangle=0$ for every $a\in R_{1},b\in R_{2}$ then each
$R_{i}$ ($i=1,2)$ is itself a root system and $W(R)=W(R_{1})\times W(R_{2})$,
a direct product. Furthermore, $W(R_{1})$ and $W(R_{2})$ act on the orthogonal
subspaces $\mathrm{span}_{\mathbb{R}}(R_{1})$ and $\mathrm{span}_{\mathbb{R}%
}(R_{2})$, respectively. In this case, the root system $R$ and the reflection
group $W(R)$ are called \emph{decomposable}. Otherwise the system and group
are \emph{indecomposable} (also called \emph{irreducible}). There is a
complete classification of indecomposable finite reflection groups.

Assume that the rank of $R$ is $d$, that is, $\mathrm{span}_{\mathbb{R}%
}(R)=\mathbb{R}^{d}$. The set of hyperplanes $H=\left\{  a^{\bot}:a\in
R_{+}\right\}  $ divides $\mathbb{R}^{d}$ into connected (open) components,
called \emph{chambers}. The order of the group equals the number of chambers
(see \cite[Theorems 1.4, 1.8]{Hum1}). Recall that $R_{+}=\{a\in R:\langle
a,a_{0}\rangle>0\}$ for some $a_{0}\in\mathbb{R}^{d}$. The connected component
of $\mathbb{R}^{d}\backslash H$ which contains $a_{0}$ is called the
\emph{fundamental chamber}. The simple roots correspond to the bounding
hyperplanes of this chamber and they form a basis of $\mathbb{R}^{d}$ (note
that the definitions of chamber and fundamental chamber extend to the
situation where $\mathrm{span}_{\mathbb{R}}(R)$ is embedded in a
higher-dimensional space). The simple reflections $s_{i}$, $i=1,\ldots,d$,
correspond to the simple roots. Let $m_{ij}$ be the order of $s_{i}s_{j}$
(clearly $m_{ii}=1$ and $m_{ij}=2$ if and only if $s_{i}s_{j}=s_{j}s_{i}$, for
$i\neq j$). The group $W(R)$ is isomorphic to the abstract group generated by
$\{s_{i}:1\leq i\leq d\}$ subject to the relations $(s_{i}s_{j})^{m_{ij}}=1$
(see \cite[Theorem 1.9]{Hum1}).

The Coxeter diagram is a graphical way of displaying the relations: it is a
graph with $d$ nodes corresponding to the simple reflections, the nodes $i$
and $j$ are joined with an edge when $m_{ij}>2$; the edge is labeled by
$m_{ij}$ when $m_{ij}>3$. The root system is indecomposable if and only if the
Coxeter diagram is connected. There follow brief descriptions of the
indecomposable root systems and the corresponding groups. The rank is
indicated by the subscript. The systems are crystallographic except for
$H_{3},H_{4}$ and $I_{2}(m)$, $m\notin\{2,3,4,6\}$.

\subsection{\label{c4rootA}Type $A_{d-1}$}

The root system is $R=\{\varepsilon_{i}-\varepsilon_{j}:i\neq j\}\subset
\mathbb{R}^{d}$. The span is $\big(\sum_{i=1}^{d}\varepsilon_{i}\big)^{\bot}$,
thus the rank is $d-1$. The reflection $s_{ij}=s_{\varepsilon_{i}%
-\varepsilon_{j}}$ interchanges the components $x_{i}$ and $x_{j}$ of each
$x\in\mathbb{R}^{d}$ and is called a \emph{transposition}, often denoted by
$(ij)$. Thus $W(R)$ is the symmetric (or permutation) group $\mathcal{S}_{d}$
of $d$ objects. Choose $a_{0}=\sum_{i=1}^{d}(d+1-i)\varepsilon_{i}$ then
$R_{+}=\{\varepsilon_{i}-\varepsilon_{j}:i<j\}$ and the simple roots are
$\{\varepsilon_{i}-\varepsilon_{i+1}:1\leq i\leq d-1\}$. The corresponding
reflections are the adjacent transpositions $s_{i}=(i,i+1)$. The structure
constants satisfy $m_{i,i+1}=3$ and $m_{ij}<3$ otherwise. The Coxeter diagram
is
\[
\circ-\circ-\circ-\cdots-\circ
\]
The alternating polynomial is
\[
a_{R}(x)=\prod\limits_{a\in R_{+}}\langle x,a\rangle=\prod\limits_{1\leq
i<j\leq d}(x_{i}-x_{j}).
\]
The fundamental chamber is $\{x:x_{1}>x_{2}>\cdots>x_{d}\}$.

\subsection{\label{c4rootB}Type $B_{d}$}

The root system is $R=\{\pm\varepsilon_{i}\pm\varepsilon_{j}:1\leq i<j\leq
d\}\cup\{\pm\varepsilon_{i}:1\leq i\leq d\}$. For $d=1$, $R=\{\pm
\varepsilon_{1}\}$ which is essentially the same as $A_{1}$. The group
$W(B_{d})$ is the full symmetry group of the hyperoctahedron $\{\pm
\varepsilon_{1},\pm\varepsilon_{2},\ldots,\pm\varepsilon_{d}\} \subset
\mathbb{R}^{d}$ (also of the hypercube) and is thus called the
\emph{hyperoctahedral group}. Its elements are the $d\times d$ generalized
permutation matrices with entries $\pm1$ (that is, each row and each column
has exactly one nonzero element $\pm1$). With the same $a_{0}$ as used for
$A_{d-1}$, the positive root system is $R_{+}=\{\varepsilon_{i}-\varepsilon
_{j},\varepsilon_{i}+\varepsilon_{j}:i<j\} \cup\{\varepsilon_{i}:1\leq i\leq
d\}$ and the simple roots are $\{\varepsilon_{i}-\varepsilon_{i+1}%
:i<d\}\cup\{\varepsilon_{d}\}$. The order of $s_{\varepsilon_{d-1}%
-\varepsilon_{d}}s_{\varepsilon_{d}}$ is 4. The Coxeter diagram is
$\circ-\circ-\circ-\cdots-\circ\overset{4}{-}\circ$

The alternating polynomial is
\[
a_{R}(x)=\prod_{i=1}^{d}x_{i} \prod\limits_{1\leq i<j\leq d}\left(  x_{i}%
^{2}-x_{j}^{2}\right)  .
\]
The fundamental chamber is $\{x:x_{1}>x_{2}>\cdots>x_{d}>0\}$.

\subsection{Types $C_{d}$ and $BC_{d}$}

The root system $C_{d}$ is $\{\pm\varepsilon_{i}\pm\varepsilon_{j}:1\leq
i<j\leq d\}\cup\{2\varepsilon_{i}:1\leq i\leq d\}$, and the system $BC_{d}$ is
$\{\pm\varepsilon_{i}\pm\varepsilon_{j}:1\leq i<j\leq d\}\cup\{2\varepsilon
_{i},\varepsilon_{i}:1\leq i\leq d\}$. The latter system is not reduced. Both
of these systems generate the same group as $W(B_{d})$. The system $C_{d}$ has
to be distinguished from $B_{d}$ because of the crystallographic condition in
Definition \ref{c4cryst}. For $d=1,2$ there is no essential distinction.

\subsection{Type $D_{d}$}

For $d\geq4$ the root system is $R=\{\pm\varepsilon_{i}\pm\varepsilon
_{j}:1\leq i<j\leq d\}$, a subset of $B_{d}$. The group $W(D_{d})$ is the
subgroup of $W(B_{d})$ fixing the polynomial $\prod_{j=1}^{d}x_{j}$. The
simple roots are $\{\varepsilon_{i}-\varepsilon_{i+1}:1\leq i<d\}\cup
\{\varepsilon_{d-1}+\varepsilon_{d}\}$ and the Coxeter diagram is
\[
\circ-\circ-\circ-\cdots-\circ{<}_{\circ}^{\circ}%
\]
The alternating polynomial is
\[
a_{R}(x)=\prod\limits_{1\leq i<j\leq d}\left(  x_{i}^{2}-x_{j}^{2}\right)
\]
and the fundamental chamber is $\{x:x_{1}>x_{2}>\cdots>|x_{d}|\}$.

\subsection{\label{c4rootF}Type $F_{4}$}

The root system is $R_{1}\cup R_{2}$ where $R_{1}=\{\pm\varepsilon_{i}%
\pm\varepsilon_{j}:1\leq i<j\leq4\}$ and $R_{2}=\{\pm\varepsilon_{i}:1\leq
i\leq4\}\cup\{\tfrac12(\pm\varepsilon_{1}\pm\varepsilon_{2}\pm\varepsilon_{3}
\pm\varepsilon_{4})\}$. Each set contains 24 roots. The group $W(F_{4})$
contains $W(B_{4})$ as a subgroup of index 3. The simple roots are
$\varepsilon_{2}-\varepsilon_{3}$, $\varepsilon_{3}-\varepsilon_{4}$,
$\varepsilon_{4}$, $\tfrac12(\varepsilon_{1}-\varepsilon_{2}-\varepsilon
_{3}-\varepsilon_{4})$ and the Coxeter diagram is
\[
\circ-\circ\overset{4}{-}\circ-\circ.
\]
With the orthogonal coordinates $y_{1}=\frac{1}{\sqrt{2}}(-x_{1}+x_{2})$,
$y_{2}=\frac{1}{\sqrt{2}}(x_{1}+x_{2})$, $y_{3}=\frac{1}{\sqrt{2}}%
(-x_{3}+x_{4})$, $y_{4}=\frac{1}{\sqrt{2}}(x_{3}+x_{4})$, the alternating
polynomial is
\[
a_{R}(x)=2^{-6}\prod\limits_{1\leq i<j\leq4}\left(  x_{i}^{2}-x_{j}%
^{2}\right)  \left(  y_{i}^{2}-y_{j}^{2}\right)  .
\]

\subsection{Type $G_{2}$}

This system is a subset $R_{1}\cup R_{2}$ of $\left\{  \sum_{i=1}^{3}%
x_{i}\varepsilon_{i}:\sum_{i=1}^{3}x_{i}=0\right\}  \subset\mathbb{R}^{3}$,
where $R_{1}=\{\pm(\varepsilon_{i}-\varepsilon_{j}):1\leq i<j\leq3\}$ and
$R_{2}=\{\pm(2\varepsilon_{i}-\varepsilon_{j}-\varepsilon_{k}):
\{i,j,k\}=\{1,2,3\}\}$. The simple roots are $\varepsilon_{1}-\varepsilon_{2}%
$, $-2\varepsilon_{1}+\varepsilon_{2}+\varepsilon_{3}$ and the Coxeter diagram
is
\[
\circ\overset{6}{-}\circ
\]

\subsection{Types $E_{6},E_{7},E_{8}$}

The root system $E_{8}$ equals $R_{1}\cup R_{2}$ where
\begin{align*}
R_{1}  &  =\{\pm\varepsilon_{i}\pm\varepsilon_{j}:1\leq i<j\leq8\},\\
R_{2}  &  =\left\{  \frac{1}{2}\sum_{i=1}^{8}(-1)^{n_{i}}\varepsilon_{i}%
:n_{i}\in\{0,1\},\;\sum_{i=1}^{8}n_{i}=0\operatorname{mod}2\right\}  ,
\end{align*}
with $\#R_{1}=112$ and $\#R_{2}=128$. The simple roots are $r_{1}=\frac{1}%
{2}\left(  \varepsilon_{1}-\sum_{i=2}^{7}\varepsilon_{i}+\varepsilon
_{8}\right)  $, $r_{2}=\varepsilon_{1}+\varepsilon_{2}$, $r_{i}=\varepsilon
_{i-1}-\varepsilon_{i-2}$ for $3\leq i\leq8$. The systems $E_{6}$ and $E_{7}$
consist of the elements of $R$ which lie in $\mathrm{span}\{r_{i}:1\leq
i\leq6\}$ and $\mathrm{span}\{r_{i}:1\leq i\leq7\}$ respectively. (Because
these systems are crystallographic, the spans are over $\mathbb{Z}$%
\thinspace.) For more details on these systems see \cite[pp. 42-43]{Hum1}.

\subsection{Type $I_{2}(m)$}

These are the dihedral systems corresponding to symmetry groups of regular
$m$-gons in $\mathbb{R}^{2}$ for $m\geq3$. Using a complex coordinate system
$z=x_{1}+\mathrm{i}\,x_{2}$ and $\overline{z}=x_{1}-\mathrm{i}\,x_{2}$, the
map $z\mapsto ze^{\mathrm{i}\theta}$ is a rotation through the angle $\theta$,
and the reflection along $(\sin\phi,-\cos\phi)$ is $z\mapsto\overline
{z}e^{2\mathrm{i}\phi}$. The reflection along $v_{(j)}=\left(  \sin\frac{\pi
j}{m},-\cos\frac{\pi j}{m}\right)  $ corresponds to $s_{j}:z\mapsto
\overline{z}e^{2\pi\mathrm{i}j/m}$, for $1\leq j\leq2m$; note that
$v_{(m+j)}=-v_{(j)}$. For $a_{0}=\left(  \cos\frac{\pi}{2m},\sin\frac{\pi}%
{2m}\right)  $, the positive roots are $\left\{  v_{(j)}:1\leq j\leq
m\right\}  $ and the simple roots are $v_{(1)}$, $v_{(m)}$. Then $s_{1}s_{m}$
maps $z$ to $ze^{2\pi\mathrm{i}/m}$, and has period $m$. The Coxeter diagram
is $\circ\overset{m}{-}\circ$ . Since $s_{j}s_{n}s_{j}=s_{2j-n}$ for any
$n,j$, there are two conjugacy classes of reflections $\{s_{2i}\}$,
$\{s_{2i+1}\}$ when $m$ is even, and one class when $m$ is odd. There are 3
special cases: the groups $W(I_{2}(3)),W(I_{2}(4)),W(I_{2}(6))$ are isomorphic
to $W(A_{2}),W(B_{2})$, $W(G_{2})$, respectively. The alternating polynomial
is a multiple of $(z^{m}-\overline{z}^{m})/\mathrm{i}$.

\subsection{Type $H_{3}$}

Let $\tau=\left(  1+\sqrt{5}\,\right)  /2$ (so $\tau^{2}=\tau+1$). Take the
positive root system to be $R_{+}=\{(2,0,0),(0,2,0),\allowbreak(0,0,2),(\tau
,\pm\tau^{-1},\pm1),(\pm1,\tau,\pm\tau^{-1}), (\tau^{-1},\pm1,\tau
),(-\tau^{-1},1,\tau),(\tau^{-1},1,-\tau)\}$, thus $\#R_{+}=15$. The root
system $R=R_{+}\cup(-R_{+})$ as a configuration in $\mathbb{R}^{3}$ is called
the \emph{icosidodecahedron}. The group $W(H_{3})$ is the symmetry group of
the icosahedron $Q_{12}=\{(0,\pm\tau,\pm1),(\pm1,0,\pm\tau),\allowbreak
(\pm\tau,\pm1,0)\}$ (12 vertices, 20 triangular faces) and of the dodecahedron
$Q_{20}=\{(0,\pm\tau^{-1},\pm\tau),\allowbreak(\pm\tau,0,\pm\tau^{-1}%
),(\pm\tau^{-1},\pm\tau,0),(\pm1,\pm1,\pm1)\}$ (20 vertices, 12 pentagonal
faces). The simple roots are $(\tau,-\tau^{-1},-1),(-1,\tau,-\tau^{-1}%
),(\tau^{-1},-1,\tau)$ and the Coxeter diagram is $\circ-\circ\overset{5}%
{-}\circ$ .

\subsection{Type $H_{4}$}

This root system has 60 positive roots, and the Coxeter diagram $\circ
-\circ-\circ\overset{5}{-}\circ$ . The group $W(H_{4})$ is the symmetry group
of the 600-cell and is sometimes called the \emph{hecatonicosahedroidal group}.

\subsection{Miscellaneous Results}

For any root system, the subgroup generated by a subset of simple reflections
(that is, the result of deleting one or more nodes from the Coxeter diagram)
is called a \emph{parabolic subgroup} of $W(R)$. More generally, any conjugate
of such a subgroup is also called parabolic. Therefore, for any $a\in
\mathbb{R}^{d}$ the stabilizer $W_{a}=\{w\in W(R):wa=a\}$ is parabolic. The
number of conjugacy classes of reflections equals the number of connected
components of the Coxeter diagram after all edges with an even label have been removed.

\section{Invariant Polynomials}

For $w\in O(d)$ and $p\in\Pi^{d}$, the space of polynomials on $\mathbb{R}%
^{d}$, let $(wp)(x)=p\left(  w^{-1}x\right)  $ (thus $((w_{1}w_{2}%
)p)(x)=(w_{1}(w_{2}p))(x)$, $w_{1},w_{2}\in O(d)$). For a finite subgroup $G$
of $O(d)$ let $\Pi^{G}$ denote the space of $G$-invariant polynomials
$\{p\in\Pi^{d}:wp=p$ for all $w\in G\}$.

When $G$ is a finite reflection group $W(R)$, $\Pi^{G}$ has an elegant
structure; there is a set of algebraically independent homogeneous generators,
whose degrees are fundamental constants associated with $R$.

\begin{theorem}
\label{c4invts}\textrm{(See \cite[Theorems 3.5, 3.9]{Hum1}.)} Suppose $R$ is a
root system in $\mathbb{R}^{d}$ then there exist $d$ algebraically independent
$W(R)$-invariant homogeneous polynomials $\{q_{j}:1\leq j\leq d\}$ of degrees
$n_{j}$, such that $\Pi^{W(R)}$ is the ring of polynomials generated by
$\{q_{j}\}$. Furthermore, $\#W(R)=n_{1}n_{2}\cdots n_{d}$ and the number of
reflections in $W(R)$ is $\sum_{j=1}^{d}(n_{j}-1)$.
\end{theorem}

For an indecomposable root system $R$ of rank $d$ the numbers $\{n_{j}:1\leq
j\leq d\}$ are called the \emph{fundamental degrees} of $W(R)$. The
coefficient of $t^{k}$ in the product $\prod_{j=1}^{d}(1+(n_{j}-1)t)$ is the
number of elements of $W(R)$ whose fixed point set is of codimension $k$ (see
\cite[Remark 3.9]{Hum1}). Here are the structural constants (see \cite[Table
3.1]{Hum1}):\newline%

\begin{tabular}
[c]{|l|l|l|l|}\hline
Type & $\#R_{+}$ & $\#W(R)$ & $n_{1},\ldots,n_{d}$\\\hline
$A_{d}$ & $\frac{d(d+1)}{2}$ & $(d+1)!$ & $2,3,\ldots,d+1$\\\hline
$B_{d}$ & $d^{2}$ & $2^{d}d!$ & $2,4,6,\ldots,2d$\\\hline
$D_{d}$ & $d(d-1)$ & $2^{d-1}d!$ & $2,4,6,\ldots,2(d-1),d$\\\hline
$G_{2}$ & $6$ & $12$ & $2,6$\\\hline
$H_{3}$ & $15$ & $120$ & $2,6,10$\\\hline
$F_{4}$ & $24$ & $1152$ & $2,6,8,12$\\\hline
$H_{4}$ & $60$ & $14400$ & $2,12,20,30$\\\hline
$E_{6}$ & $36$ & $2^{7}\times3^{4}\times5$ & $2,5,6,8,9,12$\\\hline
$E_{7}$ & $63$ & $2^{10}\times3^{4}\times5\times7$ & $2,6,8,10,12,14,18$%
\\\hline
$E_{8}$ & $120$ & $2^{14}\times3^{4}\times5^{2}\times7$ &
$2,8,12,14,18,20,24,30$\\\hline
$I_{2}\left(  m\right)  $ & $m$ & $2m$ & $2,m$\\\hline
\end{tabular}

\section{Dunkl operators}

Throughout this section $R$ denotes a reduced root system contained in
$\mathbb{R}^{d}$ and $G=W\left(  R\right)  $, a subgroup of $O\left(
d\right)  $. For $\alpha\in\mathbb{N}_{0}^{d}$ and $x\in\mathbb{R}^{d}$ let
$\left\vert \alpha\right\vert :=\sum_{i=1}^{d}\alpha_{i}$ and $x^{\alpha
}:=\prod_{i=1}^{d}x_{i}^{\alpha_{i}}$, a monomial of degree $\left\vert
\alpha\right\vert $. Let $\Pi_{n}^{d}:=\mathrm{span}_{\mathbb{F}}\left\{
x^{\alpha}:\alpha\in\mathbb{N}^{d},\left\vert \alpha\right\vert =n\right\}  $,
the space of homogeneous polynomials of degree $n$, $n\in\mathbb{N}_{0}$,
where $\mathbb{F}$ is some extension field of $\mathbb{R}$ containing the
parameter values.

\begin{definition}
A \emph{multiplicity function} on $R$ is a $G$-invariant function $\kappa$
with values in $\mathbb{R}$ or a transcendental extension of $\mathbb{Q}$,
that is, $a\in R,w\in G$ implies $\kappa\left(  wa\right)  =\kappa\left(
a\right)  $. Note $\kappa\left(  -a\right)  =\kappa\left(  a\right)  $ since
$s_{a}a=-a$.
\end{definition}

Suppose $a\in\mathbb{R}^{d}\backslash\left\{  0\right\}  $ and $w\in O\left(
d\right)  $ then $ws_{a}w^{-1}=s_{wa}$. Thus, $\kappa$ can be considered as a
function on the reflections $\left\{  s_{a}:a\in R_{+}\right\}  $ which is
constant on conjugacy classes. In the sequel $\kappa$ denotes a multiplicity
function of $G$.

For any reflection $s_{a}$ and polynomial $p\in\Pi^{d}$ the polynomial
$p\left(  x\right)  -s_{a}p\left(  x\right)  $ vanishes on $a^{\bot}$ hence is
divisible by $\left\langle x,a\right\rangle $. The gradient is denoted by
$\nabla$.

\begin{definition}
For $p\in\Pi^{d}$ and $x\in\mathbb{R}^{d}$ let%
\[
\nabla_{\kappa}p\left(  x\right)  :=\nabla p\left(  x\right)  +\sum_{v\in
R_{+}}\kappa\left(  v\right)  \frac{p\left(  x\right)  -s_{v}p\left(
x\right)  }{\left\langle x,v\right\rangle }v,
\]
and for $a\in\mathbb{R}^{d}$ let $\mathcal{D}_{a}p\left(  x\right)
:=\left\langle \nabla_{\kappa}p\left(  x\right)  ,a\right\rangle $.
$\mathcal{D}_{a}$ is a \emph{Dunkl operator}. For $1\leq i\leq d$ denote
$\mathcal{D}_{\varepsilon_{i}}$ by $\mathcal{D}_{i}$.
\end{definition}

Thus each $\mathcal{D}_{a}$ is an operator on polynomials, and maps $\Pi
_{n}^{d}$ into $\Pi_{n-1}^{d}$, $n\in\mathbb{N}_{0}$ (that is, $\mathcal{D}%
_{a}$ is homogeneous of degree $-1$). The important properties of $\left\{
\mathcal{D}_{a}:a\in\mathbb{R}^{d}\right\}  $ are $G$-covariance and commutativity.

\begin{proposition}
\label{c4wDw}Let $w\in G$ and $a\in\mathbb{R}^{d}$, then (as operators on
$\Pi^{d}$) $w^{-1}\mathcal{D}_{a}w=\mathcal{D}_{wa}$.
\end{proposition}

\begin{proof}
Let $p\in\Pi^{d}$ then%
\begin{align*}
w\mathcal{D}_{a}w^{-1}p\left(  x\right)   &  =w\left\langle \nabla
w^{-1}p\left(  x\right)  ,a\right\rangle +\frac{1}{2}w\sum_{v\in R}%
\kappa\left(  v\right)  \frac{p\left(  wx\right)  -p\left(  ws_{v}x\right)
}{\left\langle x,v\right\rangle }\left\langle v,a\right\rangle \\
&  =\left\langle \nabla p\left(  x\right)  ,wa\right\rangle +\frac{1}{2}%
\sum_{v\in R}\kappa\left(  v\right)  \frac{p\left(  x\right)  -p\left(
ws_{v}w^{-1}x\right)  }{\left\langle w^{-1}x,v\right\rangle }\left\langle
v,a\right\rangle \\
&  =\left\langle \nabla p\left(  x\right)  ,wa\right\rangle +\frac{1}{2}%
\sum_{u\in R}\kappa\left(  w^{-1}u\right)  \frac{p\left(  x\right)  -p\left(
s_{u}x\right)  }{\left\langle x,u\right\rangle }\left\langle u,wa\right\rangle
\\
&  =\mathcal{D}_{wa}p\left(  x\right)  .
\end{align*}
In the reflection part of $\mathcal{D}_{a}$ the sum over $R_{+}$ can be
replaced by $\frac{1}{2}$ of the sum over $R$. Then the summation variable
$v\in R$ is replaced by $v=w^{-1}u$.
\end{proof}

For $t\in\mathbb{R}^{d}$ let $m_{t}$ denote the multiplier operator on
$\Pi^{d}$ given by%
\[
\left(  m_{t}p\right)  \left(  x\right)  :=\left\langle t,x\right\rangle
p\left(  x\right)  .
\]
The commutator of two operators $A,B$ on $\Pi^{d}$ is $\left[  A,B\right]
:=AB-BA$. The identities $\left[  \left[  A,B\right]  ,C\right]  =\left[
\left[  A,C\right]  ,B\right]  -\left[  \left[  B,C\right]  ,A\right]  $ and
$\left[  A^{2},B\right]  =A\left[  A,B\right]  +\left[  A,B\right]  A$ are
used below.

\begin{proposition}
\label{c4Dmt}For $a,t\in\mathbb{R}^{d}$%
\[
\left[  \mathcal{D}_{a},m_{t}\right]  =\left\langle a,t\right\rangle
+2\sum_{v\in R_{+}}\kappa\left(  v\right)  \frac{\left\langle a,v\right\rangle
\left\langle t,v\right\rangle }{\left\vert v\right\vert ^{2}}s_{v}.
\]

\end{proposition}

\begin{proof}
Let $p\in\Pi^{d}$. Then $\left\langle \nabla\left(  \left\langle
t,x\right\rangle p\left(  x\right)  \right)  ,a\right\rangle -\left\langle
t,x\right\rangle \left\langle \nabla p\left(  x\right)  ,a\right\rangle
=\left\langle t,a\right\rangle p\left(  x\right)  $. Next for any $v\in R_{+}$
we have
\begin{align*}
&  \frac{\left\langle t,x\right\rangle p\left(  x\right)  -\left\langle
t,s_{v}x\right\rangle p\left(  s_{v}x\right)  }{\left\langle x,v\right\rangle
}-\left\langle t,x\right\rangle \frac{p\left(  x\right)  -p\left(
s_{v}x\right)  }{\left\langle x,v\right\rangle }\\
&  =\frac{p\left(  s_{v}x\right)  }{\left\langle x,v\right\rangle }\left(
\left\langle t,x\right\rangle -\left\langle s_{v}t,x\right\rangle \right)
=2\frac{p\left(  s_{v}x\right)  \left\langle x,v\right\rangle }{\left\langle
x,v\right\rangle \left\vert v\right\vert ^{2}}\left\langle t,v\right\rangle .
\end{align*}
This proves the formula.
\end{proof}

\begin{lemma}
For $b\in\mathbb{R}^{d}$ and $v\in R$,%
\[
\left[  s_{v},\mathcal{D}_{b}\right]  =2\frac{\left\langle b,v\right\rangle
}{\left\vert v\right\vert ^{2}}s_{v}\mathcal{D}_{v}.
\]

\end{lemma}

\begin{proof}
Indeed by Proposition \ref{c4wDw} $\left[  s_{v},\mathcal{D}_{a}\right]
=s_{v}\mathcal{D}_{a}-\mathcal{D}_{a}s_{v}=s_{v}\left(  \mathcal{D}%
_{a}-\mathcal{D}_{s_{v}a}\right)  =s_{v}\left(  2\frac{\left\langle
a,v\right\rangle }{\left\vert v\right\vert ^{2}}\mathcal{D}_{v}\right)  $.
\end{proof}

\begin{theorem}
\label{c4DD0}If $a,b\in\mathbb{R}^{d}$ then $\left[  \mathcal{D}%
_{a},\mathcal{D}_{b}\right]  =0.$
\end{theorem}

\begin{proof}
Let $t\in\mathbb{R}^{d}$; we will show $\left[  \left[  \mathcal{D}%
_{a},\mathcal{D}_{b}\right]  ,m_{t}\right]  =0$. Indeed%
\[
\left[  \left[  \mathcal{D}_{a},\mathcal{D}_{b}\right]  ,m_{t}\right]
=\left[  \left[  \mathcal{D}_{a},m_{t}\right]  ,\mathcal{D}_{b}\right]
-\left[  \left[  \mathcal{D}_{b},m_{t}\right]  ,\mathcal{D}_{a}\right]  ,
\]
and%
\begin{align*}
\left[  \left[  \mathcal{D}_{a},m_{t}\right]  ,\mathcal{D}_{b}\right]   &
=\left[  \left\langle a,t\right\rangle ,\mathcal{D}_{b}\right]  +2\sum_{v\in
R_{+}}\kappa\left(  v\right)  \frac{\left\langle a,v\right\rangle \left\langle
t,v\right\rangle }{\left\vert v\right\vert ^{2}}\left[  s_{v},\mathcal{D}%
_{b}\right] \\
&  =4\sum_{v\in R_{+}}\kappa\left(  v\right)  \frac{\left\langle
a,v\right\rangle \left\langle t,v\right\rangle \left\langle b,v\right\rangle
}{\left\vert v\right\vert ^{4}}s_{v}\mathcal{D}_{v},
\end{align*}
which is symmetric in $a,b$. The algebra generated by $\left\{  m_{t}%
:t\in\mathbb{R}^{d}\right\}  $ is $\Pi^{d}$. For any $p\in\Pi^{d}$ we have
$\left[  \mathcal{D}_{a},\mathcal{D}_{b}\right]  p=p\left[  \mathcal{D}%
_{a},\mathcal{D}_{b}\right]  1=0$.
\end{proof}

This method of proof is due to Etingof \cite[Theorem 2.15]{EtM} (the result
was first established by Dunkl \cite{Du1}). One important consequence is that
the operators $\mathcal{D}_{a}$ generate a commutative algebra.

\begin{definition}
Let $\mathcal{A}_{\kappa}$ denote the algebra of operators on $\Pi^{d}$
generated by $\left\{  \mathcal{D}_{i}:1\leq i\leq d\right\}  $. Let $\rho$
denote the homomorphism $\Pi^{d}\rightarrow\mathcal{A}_{\kappa}$ given by
$\rho p\left(  x_{1},\ldots,x_{d}\right)  =p\left(  \mathcal{D}_{1}%
,\ldots,\mathcal{D}_{d}\right)  ,p\in\Pi^{d}$.
\end{definition}

The map $\rho$ is an isomorphism.

\begin{proposition}
If $w\in G$ and $p\in\Pi^{d}$ then $\rho\left(  wp\right)  =w\rho\left(
p\right)  w^{-1}$.
\end{proposition}

\begin{proof}
It suffices to show this for first-degree polynomials. For $t\in\mathbb{R}%
^{t}$ let $p_{t}\left(  x\right)  =\left\langle t,x\right\rangle $, then $\rho
p_{t}=\sum_{i=1}^{d}t_{i}\mathcal{D}_{i}=\mathcal{D}_{t}$. Also $wp_{t}\left(
x\right)  =\left\langle t,w^{-1}x\right\rangle =\left\langle wt,x\right\rangle
=p_{wt}\left(  x\right)  $. Then $\rho\left(  wp_{t}\right)  =\rho\left(
p_{wt}\right)  =\mathcal{D}_{wt}=w\mathcal{D}_{t}w^{-1}=w\rho\left(  p\right)
w^{-1}$ by Proposition \ref{c4wDw}.
\end{proof}

There is a Laplacian-type operator in the algebra $\mathcal{A}_{\kappa}$. This
operator is an important part of the analysis of the $L^{2}$-theory associated
to the $G$-invariant weight functions.

\begin{definition}
For $p\in\Pi^{d}$ the \emph{Dunkl Laplacian} $\Delta_{\kappa}$ is given by%
\[
\Delta_{\kappa}p\left(  x\right)  :=\Delta p\left(  x\right)  +2\sum_{v\in
R_{+}}\kappa\left(  v\right)  \left\{  \frac{\left\langle \nabla p\left(
x\right)  ,v\right\rangle }{\left\langle x,v\right\rangle }-\frac{\left\vert
v\right\vert ^{2}}{2}\frac{p\left(  x\right)  -p\left(  s_{v}x\right)
}{\left\langle x,v\right\rangle ^{2}}\right\}  .
\]

\end{definition}

\begin{theorem}
\label{c4Lapmt}$\Delta_{\kappa}=\sum_{i=1}^{d}\mathcal{D}_{i}^{2}$, and
$\left[  \Delta_{\kappa},m_{t}\right]  =2\mathcal{D}_{t}$ for each
$t\in\mathbb{R}^{d}$.
\end{theorem}

\begin{proof}
We use the same method as in the proof of Theorem \ref{c4DD0}, that is, we
show $\left[  \Delta_{\kappa},m_{t}\right]  =\sum_{i=1}^{d}\left[
\mathcal{D}_{i}^{2},m_{t}\right]  $ for $t\in\mathbb{R}^{d}$. Clearly
$\Delta_{\kappa}1=0=\sum_{i=1}^{d}\mathcal{D}_{i}^{2}1$. We have%
\begin{align*}
\sum_{i=1}^{d}\left[  \mathcal{D}_{i}^{2},m_{t}\right]   &  =\sum_{i=1}%
^{d}\left\{  \mathcal{D}_{i}\left[  \mathcal{D}_{i},m_{t}\right]  +\left[
\mathcal{D}_{i},m_{t}\right]  \mathcal{D}_{i}\right\} \\
&  =\sum_{i=1}^{d}\left\{  2t_{i}\mathcal{D}_{i}+2\sum_{v\in R_{+}}%
\kappa\left(  v\right)  \frac{\left\langle t,v\right\rangle v_{i}}{\left\vert
v\right\vert ^{2}}\left(  \mathcal{D}_{i}s_{v}+s_{v}\mathcal{D}_{i}\right)
\right\} \\
&  =2\mathcal{D}_{t}+2\sum_{v\in R_{+}}\kappa\left(  v\right)  \frac
{\left\langle t,v\right\rangle }{\left\vert v\right\vert ^{2}}\left(
\mathcal{D}_{v}s_{v}+s_{v}\mathcal{D}_{v}\right)  =2\mathcal{D}_{t},
\end{align*}
because $\mathcal{D}_{v}s_{v}+s_{v}\mathcal{D}_{v}=\mathcal{D}_{v}%
s_{v}+\mathcal{D}_{s_{v}v}s_{v}$ and $\mathcal{D}_{s_{v}v}=-\mathcal{D}_{v}$.
Next $\left[  \Delta,m_{t}\right]  =2\sum_{i=1}^{d}t_{i}\frac{\partial
}{\partial x_{i}}$. For $v\in R_{+}$ let $T_{v}$ denote the operator defined
in $\left\{  \cdot\right\}  $ in the formula for $\Delta_{\kappa}$, then%
\begin{align*}
\left[  T_{v},m_{t}\right]  p\left(  x\right)   &  =\frac{\left\langle
t,v\right\rangle }{\left\langle x,v\right\rangle }p\left(  x\right)
-\frac{\left\vert v\right\vert ^{2}}{2}\frac{\left\langle x,t\right\rangle
-\left\langle s_{v}x,t\right\rangle }{\left\langle x,v\right\rangle ^{2}%
}p\left(  s_{v}x\right) \\
&  =\frac{\left\langle t,v\right\rangle }{\left\langle x,v\right\rangle
}\left(  p\left(  x\right)  -p\left(  s_{v}x\right)  \right)  ,
\end{align*}
since $\left\langle x,t\right\rangle -\left\langle s_{v}x,t\right\rangle
=\frac{2}{\left\vert v\right\vert ^{2}}\left\langle x,v\right\rangle
\left\langle t,v\right\rangle $, for any $p\in\Pi^{d}$. Thus $\left[
\Delta_{\kappa},m_{t}\right]  p=2\mathcal{D}_{t}p$.
\end{proof}

\begin{corollary}
$\Delta_{\kappa}\in\mathcal{A}_{\kappa}$ and $\left[  \Delta_{\kappa
},w\right]  =0$ for $w\in G$.
\end{corollary}

\begin{proof}
For any reflection group $p_{2}\left(  x\right)  =\left\vert x\right\vert
^{2}$ is $G$-invariant. The Theorem shows that $\Delta_{\kappa}=\rho\left(
p_{2}\right)  \in\mathcal{A}_{\kappa}$ and thus $w\Delta_{\kappa}w^{-1}%
=\Delta_{\kappa}$.
\end{proof}

There is a natural bilinear $G$-invariant form on $\Pi^{d}$ associated with
Dunkl operators. We will show that the form is symmetric, and
positive-definite when $\kappa\geq0$, that is, $\kappa\left(  v\right)  \geq0$
for all $v\in R$. The proof involves a number of ingredients.

\begin{lemma}
\label{c4sumxD}For each $n\in\mathbb{N}_{0}$ the space $\Pi_{n}^{d}$ is a
direct sum of eigenvectors of $\sum_{i=1}^{d}x_{i}\mathcal{D}_{i}$.
\end{lemma}

\begin{proof}
If $p\in\Pi^{d}$ then%
\[
\sum_{i=1}^{d}x_{i}\mathcal{D}_{i}p\left(  x\right)  =\sum_{i=1}^{d}x_{i}%
\frac{\partial}{\partial x_{i}}p\left(  x\right)  +\sum_{v\in R_{+}}%
\kappa\left(  v\right)  \left(  p\left(  x\right)  -s_{v}p\left(  x\right)
\right)  .
\]
The space $\Pi_{n}^{d}$ is a $G$-module under the action $w\mapsto\left(
p\mapsto wp\right)  $ for $w\in G,p\in\Pi_{n}^{d}$. So $\Pi_{n}^{d}=\sum
_{j=1}^{m}\oplus M_{j}$ where each $M_{j}$ is an irreducible $G$-submodule.
There are constants $c_{j}\left(  \kappa\right)  $ such that $\sum_{v\in
R_{+}}\kappa\left(  v\right)  \left(  1-s_{v}\right)  p=c_{j}\left(
\kappa\right)  p$ for each $p\in M_{j},1\leq j\leq m$, because $\sum_{v\in
R_{+}}\kappa\left(  v\right)  \left(  1-s_{v}\right)  $ is in the center of
the group algebra of $G$, (it is a sum over conjugacy classes and
$c_{j}\left(  \kappa\right)  =\sum_{v\in R_{+}}\kappa\left(  v\right)
c_{j,v}$ where $c_{j,v}\in\mathbb{Q}$ and the map $v\mapsto c_{j,v}$ is
constant on $G$-orbits in $R$). Thus $\sum_{i=1}^{d}x_{i}\mathcal{D}%
_{i}p=\left(  n+c_{j}\left(  \kappa\right)  \right)  p$ for $p\in M_{j}$.
\end{proof}

\begin{definition}
For $p,q\in\Pi^{d}$ let $\left\langle p,q\right\rangle _{\kappa}:=\rho\left(
p\right)  q\left(  x\right)  |_{x=0}$, (the polynomial $\rho\left(  p\right)
q$ is evaluated at $x=0$).
\end{definition}

\begin{theorem}
\label{c4kform}The pairing $\left\langle \cdot,\cdot\right\rangle _{\kappa}$
has the following properties:\newline(1) if $p\in\Pi_{n}^{d},q\in\Pi_{m}^{d}$
and $m\neq n$ then $\left\langle p,q\right\rangle _{\kappa}=0$;\newline(2) if
$w\in G$ and $p,q,r\in\Pi^{d}$ then $\left\langle wp,wq\right\rangle _{\kappa
}=\left\langle p,q\right\rangle _{\kappa}$ and $\left\langle rp,q\right\rangle
_{\kappa}=\left\langle p,\rho\left(  r\right)  q\right\rangle _{\kappa}%
$;\newline(3) the form is bilinear and symmetric.
\end{theorem}

\begin{proof}
If $p\in\Pi_{n}^{d},q\in\Pi_{m}^{d}$ and $m\geq n$ then $\rho\left(  p\right)
q\in\Pi_{m-n}^{d}$ and vanishes at $x=0$ if $m>n$; if $m<n$ then $\rho\left(
p\right)  q=0$; a nonzero value at $x=0$ is possible only if $m=n$ and then
$\rho\left(  p\right)  q$ is a constant. For part (2) we may assume $p,q\in
\Pi_{n}^{d}$ for some $n$ and $\left\langle p,q\right\rangle _{\kappa}%
=\rho\left(  p\right)  q$ (a constant). By Proposition \ref{c4wDw}
$\rho\left(  wp\right)  =w\rho\left(  p\right)  w^{-1}$ thus $\left\langle
wp,wq\right\rangle _{\kappa}=w\rho\left(  p\right)  w^{-1}wq\allowbreak
=w\rho\left(  p\right)  q=\left\langle p,q\right\rangle _{\kappa}$ (because
$w1=1$). That $\left\langle rp,q\right\rangle _{\kappa}=\left\langle
p,\rho\left(  r\right)  q\right\rangle _{\kappa}$ follows easily from the
definition. Use induction on part (3): for constants $p_{1},p_{2}$ the form
equals $\left\langle p_{1}1,p_{2}\right\rangle _{\kappa}=p_{1}p_{2}$; assume
the form is symmetric on $\sum_{m=0}^{n}\Pi_{m}^{d}$ for some $n$ and let
$p,q\in\Pi_{n+1}^{d}$. Using Lemma \ref{c4sumxD} suppose $p$ and $q$ are
eigenfunctions of $\sum_{v\in R_{+}}\kappa\left(  v\right)  \left(
1-s_{v}\right)  $ with eigenvalues $c_{1}\left(  \kappa\right)  $ and
$c_{2}\left(  \kappa\right)  $ respectively then%
\begin{align*}
c_{1}\left(  \kappa\right)  \left\langle p,q\right\rangle _{\kappa}  &
=\sum_{v\in R_{+}}\kappa\left(  v\right)  \left\langle \left(  1-s_{v}\right)
p,q\right\rangle _{\kappa}=\sum_{v\in R_{+}}\kappa\left(  v\right)
\left\langle p,\left(  1-s_{v}\right)  q\right\rangle \\
&  =c_{2}\left(  \kappa\right)  \left\langle p,q\right\rangle _{\kappa}.
\end{align*}
So $c_{1}\left(  \kappa\right)  \neq c_{2}\left(  \kappa\right)  $ implies
$\left\langle p,q\right\rangle _{\kappa}=0$ (a symmetric relation). Now assume
$c_{1}\left(  \kappa\right)  =c_{2}\left(  \kappa\right)  $, then
\begin{align*}
\left(  n+c_{1}\left(  \kappa\right)  \right)  \left\langle p,q\right\rangle
_{\kappa}  &  =\left\langle \sum_{i=1}^{d}x_{i}\mathcal{D}_{i}p,q\right\rangle
=\sum_{i=1}^{d}\left\langle \mathcal{D}_{i}p,\mathcal{D}_{i}q\right\rangle
_{\kappa}\\
&  =\sum_{i=1}^{d}\left\langle \mathcal{D}_{i}q,\mathcal{D}_{i}p\right\rangle
_{\kappa}=\left\langle \sum_{i=1}^{d}x_{i}\mathcal{D}_{i}q,p\right\rangle \\
&  =\left(  n+c_{1}\left(  \kappa\right)  \right)  \left\langle
q,p\right\rangle _{\kappa},
\end{align*}
where the inductive hypothesis was used to imply $\left\langle \mathcal{D}%
_{i}p,\mathcal{D}_{i}q\right\rangle _{\kappa}=\left\langle \mathcal{D}%
_{i}q,\mathcal{D}_{i}p\right\rangle _{\kappa}$. For fixed $p,q$ $\left(
\left\langle p,q\right\rangle _{\kappa}-\left\langle q,p\right\rangle
_{\kappa}\right)  $ is a polynomial in the values of $\kappa$ and
$n+c_{1}\left(  \kappa\right)  =0$ defines a hyperplane in the parameter space
(the dimension equals the number of $G$-orbits in $R$). Thus $\left\langle
p,q\right\rangle _{\kappa}=\left\langle q,p\right\rangle _{\kappa}$ for all
$\kappa$.
\end{proof}

\begin{corollary}
\label{c4!form}Suppose $\left\langle \cdot,\cdot\right\rangle _{1}$ is a
bilinear symmetric form on $\Pi^{d}$ such that $\left\langle 1,1\right\rangle
_{1}=1$ and $\left\langle x_{i}p,q\right\rangle _{1}=\left\langle
p,\mathcal{D}_{i}q\right\rangle _{1}$ for all $p,q\in\Pi^{d}$ and $1\leq i\leq
d$, then $\left\langle \cdot,\cdot\right\rangle _{1}=\left\langle \cdot
,\cdot\right\rangle _{\kappa}$.
\end{corollary}

\begin{proof}
Let $\alpha\in\mathbb{N}_{0}^{d}$ and $\left\vert \alpha\right\vert =n$ for
some $n\geq0$. Let $q\in\sum_{m=0}^{n-1}\Pi_{m}^{d}$. By hypothesis
$\left\langle x^{\alpha},q\right\rangle _{1}=\left\langle 1,\rho\left(
x^{\alpha}\right)  q\right\rangle _{1}=0$. If $q\in\Pi_{n}^{d}$ then
$\left\langle x^{\alpha},q\right\rangle _{1}=\left\langle 1,\rho\left(
x^{\alpha}\right)  q\right\rangle _{1}\allowbreak=\left\langle 1,\rho\left(
x^{\alpha}\right)  q\right\rangle _{\kappa}=\left\langle x^{\alpha
},q\right\rangle _{\kappa}$. By linearity and symmetry the proof is complete.
\end{proof}

\subsection{The Gaussian form}

The operator $e^{\Delta_{\kappa}/2}$ maps $\Pi_{n}^{d}$ into $\sum_{m=0}%
^{n}\Pi_{m}^{d}$ and its inverse is $e^{-\Delta_{\kappa}/2}$.

\begin{definition}
The \emph{Gaussian form} on $\Pi^{d}$ is given by%
\[
\left\langle p,q\right\rangle _{g}:=\left\langle e^{\Delta_{\kappa}%
/2}p,e^{\Delta_{\kappa}/2}q\right\rangle _{\kappa}.
\]

\end{definition}

\begin{proposition}
The Gaussian form is symmetric, bilinear and if $p,q\in\Pi^{d}$ then
$\left\langle wp,wq\right\rangle _{g}=\left\langle p,q\right\rangle _{g}$ for
all $w\in G$, $\left\langle \mathcal{D}_{i}p,q\right\rangle _{g}=\left\langle
p,\left(  x_{i}-\mathcal{D}_{i}\right)  q\right\rangle _{g}$ and $\left\langle
x_{i}p,q\right\rangle _{g}=\left\langle p,x_{i}q\right\rangle _{g}$ for $1\leq
i\leq d$.
\end{proposition}

\begin{proof}
The first claim follows from the property $w\Delta_{\kappa}=\Delta_{\kappa}w$.
By Theorem \ref{c4Lapmt} $\left[  \Delta_{\kappa},m_{\varepsilon_{i}}\right]
=2\mathcal{D}_{i}$ ( $m_{\varepsilon_{i}}$ denotes multiplication by $x_{i}$).
Using the general formula $\left[  A^{n},B\right]  =A\left[  A^{n-1},B\right]
+\left[  A,B\right]  A^{n-1}$ repeatedly shows $\left[  \Delta_{\kappa}%
^{n},m_{\varepsilon_{i}}\right]  =2n\mathcal{D}_{i}\Delta_{\kappa}^{n-1}$.
This implies $e^{\Delta_{\kappa}/2}x_{i}p\left(  x\right)  -x_{i}%
e^{\Delta_{\kappa}/2}p\left(  x\right)  =\mathcal{D}_{i}e^{\Delta_{\kappa}%
/2}p\left(  x\right)  $ for any $p\in\Pi^{d}$. Thus%
\begin{align*}
\left\langle \mathcal{D}_{i}p,q\right\rangle _{g}  &  =\left\langle
\mathcal{D}_{i}e^{\Delta_{\kappa}/2}p,e^{\Delta_{\kappa}/2}q\right\rangle
_{\kappa}=\left\langle e^{\Delta_{\kappa}/2}p,x_{i}e^{\Delta_{\kappa}%
/2}q\right\rangle _{\kappa}\\
&  =\left\langle e^{\Delta_{\kappa}/2}p,e^{\Delta_{\kappa}/2}\left(
x_{i}-\mathcal{D}_{i}\right)  q\right\rangle _{\kappa}=\left\langle p,\left(
x_{i}-\mathcal{D}_{i}\right)  q\right\rangle _{g}.
\end{align*}
Finally $\left\langle p,x_{i}q\right\rangle _{g}=$ $\left\langle
\mathcal{D}_{i}p,q\right\rangle _{g}+\left\langle p,\mathcal{D}_{i}%
q\right\rangle _{g}$, which is symmetric in $p,q$.
\end{proof}

Thus the Gaussian form satisfies $\left\langle p,q\right\rangle _{g}%
=\left\langle 1,pq\right\rangle _{g}$. This suggests that there may be an
integral formula; this is indeed the situation when $\kappa\geq0$. The
properties in the Proposition imply a uniqueness result by \textquotedblleft
reading the proof backwards\textquotedblright; set $\left\langle
p,q\right\rangle _{1}=\left\langle e^{-\Delta_{\kappa}/2}p,e^{-\Delta_{\kappa
}/2}q\right\rangle _{g}$ and use Corollary \ref{c4!form}.

\begin{definition}
For $\kappa\geq0$ the fundamental $G$-\emph{invariant weight function} is%
\[
w_{\kappa}\left(  x\right)  :=\prod_{v\in R_{+}}\left\vert \left\langle
x,v\right\rangle \right\vert ^{2\kappa\left(  v\right)  },x\in\mathbb{R}^{d}.
\]

\end{definition}

The G-invariance is a consequence of the definition of multiplicity functions.

\begin{definition}
\label{c4Gauss2}For $p,q\in\Pi^{d}$ and $\kappa\geq0$ let%
\[
\left\langle p,q\right\rangle _{2}:=\frac{c_{\kappa}}{\left(  2\pi\right)
^{d/2}}\int_{\mathbb{R}^{d}}p\left(  x\right)  q\left(  x\right)  w_{\kappa
}\left(  x\right)  e^{-\left\vert x\right\vert ^{2}/2}dx,
\]
where $c_{\kappa}$ is the normalizing constant resulting in $\left\langle
1,1\right\rangle _{2}=1$.
\end{definition}

The constant $c_{\kappa}$ is related to the Macdonald-Mehta integral. This
will be discussed below. It is not involved in the following.

\begin{theorem}
\label{c4Gint}If $p,q\in\Pi^{d},1\leq i\leq d$ then $\left\langle
\mathcal{D}_{i}p,q\right\rangle _{2}=\left\langle p,\left(  x_{i}%
-\mathcal{D}_{i}\right)  q\right\rangle _{2}$. The forms $\left\langle
\cdot,\cdot\right\rangle _{2}$ and $\left\langle \cdot,\cdot\right\rangle
_{g}$ are equal when $\kappa\geq0$.
\end{theorem}

\begin{proof}
Let $p,q\in\Pi^{d},1\leq i\leq d$. Note
\[
\frac{\partial}{\partial x_{i}}\left(  w_{\kappa}\left(  x\right)
e^{-\left\vert x\right\vert ^{2}/2}\right)  =\left(  -x_{i}+2\sum_{v\in R_{+}%
}\frac{\kappa\left(  v\right)  v_{i}}{\left\langle x,v\right\rangle }\right)
w_{\kappa}\left(  x\right)  e^{-\left\vert x\right\vert ^{2}/2};
\]
(in the special case $0<\kappa\left(  v\right)  <\frac{1}{2}$ the formula is
valid provided $\left\langle x,v\right\rangle \neq0$.) It suffices to show the
following integral vanishes:%
\begin{align*}
&  \int_{\mathbb{R}^{d}}\left(  q\mathcal{D}_{i}p+p\mathcal{D}_{i}%
q-x_{i}pq\right)  w_{\kappa}e^{-\left\vert x\right\vert ^{2}/2}dx\\
&  =\int_{\mathbb{R}^{d}}\frac{\partial}{\partial x_{i}}\left(  pq\right)
w_{\kappa}e^{-\left\vert x\right\vert ^{2}/2}dx-\int_{\mathbb{R}^{d}}%
x_{i}pqw_{\kappa}e^{-\left\vert x\right\vert ^{2}/2}dx\\
&  +\sum_{v\in R_{+}}\kappa\left(  v\right)  v_{i}\int_{\mathbb{R}^{d}}%
\frac{2p\left(  x\right)  q\left(  x\right)  -p\left(  s_{v}x\right)  q\left(
x\right)  -p\left(  x\right)  q\left(  s_{v}x\right)  }{\left\langle
x,v\right\rangle }w_{\kappa}e^{-\left\vert x\right\vert ^{2}/2}dx\\
&  =-\sum_{v\in R_{+}}\kappa\left(  v\right)  v_{i}\int_{\mathbb{R}^{d}}%
\frac{p\left(  s_{v}x\right)  q\left(  x\right)  +p\left(  x\right)  q\left(
s_{v}x\right)  }{\left\langle x,v\right\rangle }w_{\kappa}e^{-\left\vert
x\right\vert ^{2}/2}dx=0,
\end{align*}
because in each term the integrand is odd under the action of a reflection
$s_{v}$. Integration by parts and exponential decay shows $\int_{\mathbb{R}%
^{d}}\frac{\partial}{\partial x_{i}}\left(  pq\right)  w_{\kappa
}e^{-\left\vert x\right\vert ^{2}/2}dx=-\int_{\mathbb{R}^{d}}pq\frac{\partial
}{\partial x_{i}}\left(  w_{\kappa}e^{-\left\vert x\right\vert ^{2}/2}\right)
dx$. The terms in the sum over $R_{+}$ have the singularity $\left\vert
\left\langle x,v\right\rangle \right\vert ^{2\kappa\left(  v\right)  -1}$
which is integrable for $\kappa\left(  v\right)  >0$ (the terms with
$\kappa\left(  v\right)  =0$ do not appear).
\end{proof}

As a consequence of the Theorem, and of the invertibility of $e^{\Delta
_{\kappa}/2}$ on polynomials it follows that the forms $\left\langle
\cdot,\cdot\right\rangle _{\kappa}$ and $\left\langle \cdot,\cdot\right\rangle
_{g}$ are positive-definite when $\kappa\geq0$.

There is an elegant formula for $c_{\kappa}$ when there is only one $G$-orbit
in $R$. Recall the discriminant $a_{R}\left(  x\right)  =\prod_{v\in R_{+}%
}\left\langle x,v\right\rangle $.

\begin{theorem}
Suppose $G$ has just one conjugacy class of reflections, $\left\vert
v\right\vert ^{2}=2$ for each $v\in R$ and $\kappa_{0}>0$ , then%
\[
\left(  2\pi\right)  ^{-d/2}\int_{\mathbb{R}^{d}}\left\vert a_{R}\left(
x\right)  \right\vert ^{2\kappa_{0}}e^{-\left\vert x\right\vert ^{2}%
/2}dx=\prod_{i=1}^{d}\frac{\Gamma\left(  1+n_{i}\kappa_{0}\right)  }%
{\Gamma\left(  1+\kappa_{0}\right)  },
\]
where $\left\{  n_{i}:1\leq i\leq d\right\}  $ is the set of fundamental
degrees of $G$ (see Theorem \ref{c4invts}).
\end{theorem}

(If the rank of $G$ is less than $d$ then some of the degrees equal $1$.) The
integral for the symmetric group can be deduced from Selberg's integral
formula. Macdonald conjectured a formula for $c_{\kappa}$ for any reflection
group. Opdam \cite{Op1} proved the formula for all cases, except for a
constant (independent of $\kappa)$ multiple for $H_{3}$ and $H_{4}$; the
constant was verified by Garvan with a computer-assisted proof. Etingof
\cite{Et1} extended Opdam's method to a general proof for the one-class type.

\begin{corollary}
With the notation as in the Theorem and $\kappa\left(  v\right)  =\kappa_{0}$
for all $v\in R$%
\[
\left\langle a_{R},a_{R}\right\rangle _{\kappa}=\left\langle a_{R}%
,a_{R}\right\rangle _{g}=\#G\prod_{i=1}^{d}\prod_{j=1}^{n_{i}-1}\left(
j+n_{i}\kappa_{0}\right)  .
\]

\end{corollary}

\begin{proof}
The polynomial $\Delta_{\kappa}a_{R}=0$ because it is alternating and of
degree $\#R_{+}-2$, and $a_{R}$ is the minimal degree nonzero alternating
polynomial. Thus $e^{\Delta_{\kappa}/2}a_{R}=a_{R}$. From the definition of
$c_{\kappa}$ it follows that%
\begin{align*}
\left\langle a_{R},a_{R}\right\rangle _{g}  &  =\frac{c_{\kappa_{0}}%
}{c_{\kappa_{0}+1}}=\left(  1+\kappa_{0}\right)  ^{-d}\prod_{i=1}^{d}%
\frac{\Gamma\left(  1+n_{i}+n_{i}\kappa_{0}\right)  }{\Gamma\left(
1+n_{i}\kappa_{0}\right)  }\\
&  =\left(  1+\kappa_{0}\right)  ^{-d}\prod_{i=1}^{d}\prod_{j=1}^{n_{i}%
}\left(  j+n_{i}\kappa_{0}\right)  =\prod_{i=1}^{d}n_{i}\prod_{j=1}^{n_{i}%
-1}\left(  j+n_{i}\kappa_{0}\right)  ,
\end{align*}
and $\prod_{i=1}^{d}n_{i}=\#G$.
\end{proof}

The formula is valid for all $\kappa_{0}\in\mathbb{C}$ because it is a
polynomial identity.

The indecomposable reflection groups with two classes of reflections consist
of $I_{2}\left(  2m\right)  $ with $m\geq2$, $B_{d}$ and $F_{4}$. To get a
convenient expression for the dihedral group $I_{2}\left(  2m\right)  $ let
$z=x_{1}+\mathrm{i}x_{2}$ and denote the values of $\kappa$ by $\kappa_{0}$
and $\kappa_{1}$; write the weight function as $w_{\kappa}\left(  x\right)
=\left\vert z^{m}-\overline{z}^{m}\right\vert ^{2\kappa_{0}}\left\vert
z^{m}+\overline{z}^{m}\right\vert ^{2\kappa_{1}}$. Then%
\[
\frac{1}{2\pi}\int_{\mathbb{R}^{2}}w_{\kappa}\left(  x\right)  e^{-\left\vert
x\right\vert ^{2}/2}dx=2^{m\left(  \kappa_{0}+\kappa_{1}\right)  }\frac
{\Gamma\left(  1+2\kappa_{0}\right)  \Gamma\left(  1+2\kappa_{1}\right)
\Gamma\left(  1+m\left(  \kappa_{0}+\kappa_{1}\right)  \right)  }%
{\Gamma\left(  1+\kappa_{0}\right)  \Gamma\left(  1+\kappa_{1}\right)
\Gamma\left(  1+\kappa_{0}+\kappa_{1}\right)  }.
\]

For the hyperoctahedral group $B_{d}$ (see Section \ref{c4rootB} for the root
system) the integral is
\begin{align*}
&  \left(  2\pi\right)  ^{-d/2}\int_{\mathbb{R}^{d}}\prod_{i=1}^{d}\left\vert
x_{i}\right\vert ^{2\kappa_{1}}\prod_{1\leq i<j\leq d}\left\vert x_{i}%
^{2}-x_{j}^{2}\right\vert ^{2\kappa_{0}}e^{-\left\vert x\right\vert ^{2}%
/2}dx\\
&  =2^{d\left(  \left(  d-1\right)  \kappa_{0}+\kappa_{1}\right)  }\prod
_{i=1}^{d}\frac{\Gamma\left(  1+i\kappa_{0}\right)  \Gamma\left(  \left(
i-1\right)  \kappa_{0}+\kappa_{1}+\frac{1}{2}\right)  }{\Gamma\left(
1+\kappa_{0}\right)  \Gamma\left(  \frac{1}{2}\right)  }.
\end{align*}
In the formula $\Gamma\left(  \frac{1}{2}\right)  $ is used in place of
$\sqrt{\pi}$ for the sake of appearance. The formula can be derived from
Selberg's integral.

Using the notation of Section \ref{c4rootF} for the group $F_{4}$ (a special
case of Opdam's result)%
\begin{align*}
&  \frac{1}{4\pi^{2}}\int_{\mathbb{R}^{4}}\prod_{1\leq i<j\leq4}\left\vert
x_{i}^{2}-x_{j}^{2}\right\vert ^{2\kappa_{1}}\prod_{1\leq i<j\leq4}\left\vert
y_{i}^{2}-y_{j}^{2}\right\vert ^{2\kappa_{2}}e^{-\left\vert x\right\vert
^{2}/2}dx\\
&  =2^{12\left(  \kappa_{1}+\kappa_{2}\right)  }\frac{\Gamma\left(
2\kappa_{1}+\kappa_{2}+\frac{1}{2}\right)  \Gamma\left(  \kappa_{1}%
+2\kappa_{2}+\frac{1}{2}\right)  \Gamma\left(  3\kappa_{1}+3\kappa_{2}%
+\frac{1}{2}\right)  }{\Gamma\left(  \frac{1}{2}\right)  ^{3}}\\
&  \times\frac{\Gamma\left(  4\kappa_{1}+4\kappa_{2}+1\right)  }{\Gamma\left(
\kappa_{1}+\kappa_{2}+1\right)  }\prod_{i=1}^{2}\frac{\Gamma\left(
2\kappa_{i}+1\right)  \Gamma\left(  3\kappa_{i}+1\right)  }{\Gamma\left(
\kappa_{i}+1\right)  ^{2}}.
\end{align*}
The formula agrees with the simpler single-class result when $\kappa
_{1}=\kappa_{2}$, and with the $D_{4}$ value when $\kappa_{2}=0$ (by use of
the Gamma function duplication formula).

\section{Harmonic polynomials}

The Gaussian form is an important part of our analysis. Accordingly the
polynomials in the kernel of $\Delta_{\kappa}$ have properties relevant to the
two forms.

\begin{definition}
Let $\mathcal{H}_{\kappa}:=\left\{  p\in\Pi^{d}:\Delta_{\kappa}p=0\right\}  $
and $\mathcal{H}_{\kappa,n}:=\mathcal{H}_{\kappa}\cap\Pi_{n}^{d}$ for
$n=0,1,2,\ldots$. These are the spaces of \emph{harmonic and harmonic
homogeneous polynomials}, respectively. Let $\gamma_{\kappa}:=\sum_{v\in
R_{+}}\kappa\left(  v\right)  $ ($w_{\kappa}$ is positively homogeneous of
degree $2\gamma_{\kappa}$).
\end{definition}

For convenience we let $\left\vert x\right\vert ^{2}$ denote both the
polynomial in $\Pi_{2}^{d}$ and the corresponding multiplier operator. We will
show that $\Pi_{n}^{d}=\sum_{j=0}^{\left\lfloor n/2\right\rfloor }%
\oplus\left\vert x\right\vert ^{2j}\mathcal{H}_{\kappa,n-2j}$ for each
$n\geq2$ provided that $\gamma_{\kappa}+\frac{d}{2}\notin-\mathbb{N}_{0}$.
Trivially $\Pi_{n}^{d}=\mathcal{H}_{\kappa,n}$ for $n=0,1$. Note that the
proof can not use any nonsingularity property of the form $\left\langle
\cdot,\cdot\right\rangle _{\kappa}$.

\begin{lemma}
\label{c4Lapr2m}For $m=1,2,3,\ldots$%
\[
\left[  \Delta_{\kappa},\left\vert x\right\vert ^{2m}\right]  =2m\left\vert
x\right\vert ^{2\left(  m-1\right)  }\left(  2m-2+d+2\gamma_{\kappa}%
+2\sum_{i=1}^{d}x_{i}\frac{\partial}{\partial x_{i}}\right)  .
\]

\end{lemma}

\begin{corollary}
If $m,n,k=1,2,3,\ldots$ and $p\in\mathcal{H}_{\kappa,n}$ then\newline1)
$\Delta_{\kappa}\left(  \left\vert x\right\vert ^{2m}p\left(  x\right)
\right)  =2m\left(  2m-2+d+2\gamma_{\kappa}+2n\right)  \left\vert x\right\vert
^{2m-2}p\left(  x\right)  $,\newline2) $\Delta_{\kappa}^{k}\left(  \left\vert
x\right\vert ^{2m}p\left(  x\right)  \right)  =4^{k}\left(  -m\right)
_{k}\left(  1-m-d/2-\gamma_{\kappa}-n\right)  _{k}\left\vert x\right\vert
^{2m-2k}p\left(  x\right)  $.
\end{corollary}

Part 2) implies that $\Delta_{\kappa}^{k}\left(  \left\vert x\right\vert
^{2m}p\left(  x\right)  \right)  =0$ if $k>m$. This leads to an orthogonality relation.

\begin{proposition}
\label{c4hmperp}Suppose $m,k\leq\frac{n}{2},n\geq2,p\in\mathcal{H}%
_{\kappa,n-2k}$, and $q\in\mathcal{H}_{\kappa,n-2m}$; if $m\neq k$ then
$\left\langle \left\vert x\right\vert ^{2k}p,\left\vert x\right\vert
^{2m}q\right\rangle _{\kappa}=0$.
\end{proposition}

\begin{proof}
By the symmetry of the form we may assume $k>m$; then $\left\langle \left\vert
x\right\vert ^{2k}p,\left\vert x\right\vert ^{2m}q\right\rangle _{\kappa
}=\left\langle p,\Delta_{\kappa}^{k}\left\vert x\right\vert ^{2m}%
q\right\rangle _{\kappa}=0$.
\end{proof}

\begin{definition}
Suppose $\gamma_{\kappa}+\frac{d}{2}\neq0,-1,-2,\ldots$and $n=2,3,4,\ldots
$then let%
\[
\pi_{\kappa,n}:=\sum_{j=0}^{\left\lfloor n/2\right\rfloor }\frac{1}%
{4^{j}j!\left(  -\gamma_{\kappa}-n+2-d/2\right)  _{j}}\left\vert x\right\vert
^{2j}\Delta_{\kappa}^{j};
\]
for $n=0,1$ let $\pi_{\kappa,n}:=I$.
\end{definition}

The following is a version of Dixon's summation theorem (see \cite[16.4.4]%
{NIST}).

\begin{lemma}
\label{c43F2}Suppose $k\in\mathbb{N}_{0}$ and $a,b$ satisfy $a+1,a-b+1\notin
-\mathbb{N}_{0}$ then%
\[
_{3}F_{2}\left(
\genfrac{}{}{0pt}{}{-k,a,b}{k+a+1,a-b+1}%
;1\right)  =\frac{\left(  a+1\right)  _{k}\left(  \frac{a}{2}-b+1\right)
_{k}}{\left(  \frac{a}{2}+1\right)  _{k}\left(  a-b+1\right)  _{k}}%
\]

\end{lemma}

\begin{proposition}
\label{c4harmex}If $\gamma_{\kappa}+\frac{d}{2}\notin-\mathbb{N}_{0}$ and
$p\in\Pi_{n}^{d},n=2,3,4,\ldots$then $\pi_{\kappa,n}p\in\mathcal{H}_{\kappa
,n}$; if $p\in\mathcal{H}_{\kappa,n}$ then $p=\pi_{\kappa,n}p$, that is,
$\pi_{\kappa,n}$ is a projection. Furthermore%
\[
p=\sum_{j=0}^{\left\lfloor n/2\right\rfloor }\frac{1}{4^{j}j!\left(
\gamma_{\kappa}+d/2+n-2j\right)  _{j}}\left\vert x\right\vert ^{2j}\pi
_{\kappa,n-2j}\left(  \Delta_{\kappa}^{j}p\right)  .
\]

\end{proposition}

\begin{proof}
The first part is a consequence of Lemma \ref{c4Lapr2m}. The proof of the
expansion formula depends on Lemma \ref{c43F2}. Set $a=1-n-\gamma_{\kappa
}-\frac{d}{2}$ and $b=\frac{a}{2}+1$ then the coefficient of $\left(
4^{k}k!\right)  ^{-1}\left\vert x\right\vert ^{2k}\Delta_{\kappa}^{k}$ in the
right side is%
\begin{gather*}
\sum_{j=0}^{k}\binom{k}{j}\frac{1}{\left(  1-a-2j\right)  _{j}\left(
a+1+2j\right)  _{k-j}}=\sum_{j=0}^{k}\frac{\left(  -k\right)  _{j}\left(
a+1\right)  _{2j}}{j!\left(  a+j\right)  _{j}\left(  a+1\right)  _{k+j}}\\
=\frac{1}{\left(  a+1\right)  _{k}}\sum_{j=0}^{k}\frac{\left(  -k\right)
_{j}\left(  a+1\right)  _{2j}\left(  a\right)  _{j}}{j!\left(  a\right)
_{2j}\left(  a+1+k\right)  _{j}}\\
=\frac{1}{\left(  a+1\right)  _{k}}\sum_{j=0}^{k}\frac{\left(  -k\right)
_{j}\left(  \frac{a}{2}+1\right)  _{j}\left(  a\right)  _{j}}{j!\left(
\frac{a}{2}\right)  _{j}\left(  a+1+k\right)  _{j}},
\end{gather*}
which vanishes for $k\geq1$ because of the term $\left(  \frac{a}%
{2}-b+1\right)  _{k}=\left(  0\right)  _{k}$ in the summation formula Lemma
\ref{c43F2}; the transformation $\frac{\left(  a+1\right)  _{2j}}{\left(
a\right)  _{2j}}=\frac{a+2j}{a}=\frac{\left(  a/2+1\right)  _{j}}{\left(
a/2\right)  _{j}}$ was used in the last step. The sum equals $1$ when $k=0.$
The identity holds for generic $\gamma_{\kappa}$, and the hypothesis
$\gamma_{\kappa}+\frac{d}{2}\notin-\mathbb{N}_{0}$ insures that the terms in
$\pi_{\kappa,n-2j}$ are well-defined.
\end{proof}

This establishes the validity of $\Pi_{n}^{d}=\sum_{j=0}^{\left\lfloor
n/2\right\rfloor }\oplus\left\vert x\right\vert ^{2j}\mathcal{H}_{\kappa
,n-2j}$, provided $\gamma_{\kappa}+\frac{d}{2}\notin-\mathbb{N}_{0}$. (Argue
by induction that $\mathcal{H}_{\kappa,n}\cap\left\vert x\right\vert ^{2}%
\Pi_{n-2}^{d}=\left\{  0\right\}  $.) To transfer these results to the
Gaussian form let $p\in\mathcal{H}_{\kappa,n},s\in\mathbb{R}$ and evaluate%
\begin{align}
e^{s\Delta_{\kappa}}\left(  \left\vert x\right\vert ^{2m}p\left(  x\right)
\right)   &  =\sum_{j=0}^{m}\frac{1}{j!}s^{j}\Delta_{\kappa}^{j}\left(
\left\vert x\right\vert ^{2m}p\left(  x\right)  \right) \label{LaguerreF}\\
&  =\sum_{j=0}^{m}\frac{1}{j!}\left(  4s\right)  ^{j}\left(  -m\right)
_{j}\left(  1-m-d/2-\gamma_{\kappa}-n\right)  _{j}\left\vert x\right\vert
^{2m-2j}p\left(  x\right) \nonumber\\
&  =\left(  4s\right)  ^{m}m!L_{m}^{\left(  \alpha\right)  }\left(
-\frac{\left\vert x\right\vert ^{2}}{4s}\right)  p\left(  x\right)  ,\nonumber
\end{align}
where $\alpha=\gamma_{\kappa}+\frac{d}{2}+n-1$ and $L_{m}^{\left(
\alpha\right)  }$ denotes the Laguerre polynomial of degree $m$ and index
$\alpha$; it is defined provided $\alpha+1\notin-\mathbb{N}_{0}$ and part of
an orthogonal family of polynomials if $\alpha>-1$. Here we use $s=-\frac
{1}{2}$.

We list the conditions on $\gamma_{\kappa}$ specialized to some reflection groups:

\begin{itemize}
\item $I_{2}\left(  2m\right)  $: $m\left(  \kappa_{0}+\kappa_{1}\right)
+1\notin-\mathbb{N}_{0}$;

\item $A_{d-1}\subset\mathbb{R}^{d}$: $\frac{d}{2}\left(  \left(  d-1\right)
\kappa+1\right)  \notin-\mathbb{N}_{0}$;

\item $B_{d}$: $d\left(  \left(  d-1\right)  \kappa_{0}+\kappa_{1}+\frac{1}%
{2}\right)  \notin-\mathbb{N}_{0}$.
\end{itemize}

When $\kappa\geq0$ the Gaussian form can be related to an integral over the
unit sphere, and there is an analog of the spherical harmonics. Let $\omega$
denote the normalized rotation-invariant measure on the surface of the unit
sphere $S_{d-1}=\left\{  x\in\mathbb{R}^{d}:\left\vert x\right\vert
=1\right\}  $. Suppose $f$ is continuous and integrable over $\mathbb{R}^{d}$
then
\[
\int_{\mathbb{R}^{d}}f\left(  x\right)  dx=\frac{2\pi^{d/2}}{\Gamma\left(
d/2\right)  }\int_{0}^{\infty}r^{d-1}dr\int_{S_{d-1}}f\left(  ru\right)
d\omega\left(  u\right)  .
\]
The constant multiplier is evaluated by setting $f\left(  x\right)
=e^{-\left\vert x\right\vert ^{2}/2}$. Now suppose $f$ is positively
homogeneous of degree $\beta$ (that is, $f\left(  tx\right)  =t^{\beta
}f\left(  x\right)  $ for $t>0$) and $\beta+d>1$ then%
\[
\left(  2\pi\right)  ^{-d/2}\int_{\mathbb{R}^{d}}f\left(  x\right)
e^{-\left\vert x\right\vert ^{2}/2}dx=2^{\beta/2}\frac{\Gamma\left(
\frac{\beta+d}{2}\right)  }{\Gamma\left(  \frac{d}{2}\right)  }\int_{S_{d-1}%
}f\left(  u\right)  d\omega\left(  u\right)  .
\]
To normalize the measure $w_{\kappa}d\omega$ set $f=w_{\kappa}$ (with
$\beta=2\gamma_{\kappa}$) and let%
\[
c_{\kappa,S}^{-1}:=\int_{S_{d-1}}w_{\kappa}d\omega=2^{-\gamma_{\kappa}}%
\frac{\Gamma\left(  \frac{d}{2}\right)  }{\Gamma\left(  \gamma_{\kappa}%
+\frac{d}{2}\right)  }c_{\kappa}^{-1}.
\]
Observe that the condition $\gamma_{\kappa}+\frac{d}{2}\notin-\mathbb{N}_{0}$
appears again.

\begin{proposition}
\label{c4sphint}Suppose $p\in\mathcal{H}_{\kappa,n}$ and $q\in\mathcal{H}%
_{\kappa,m}$,\newline1) if $m\neq n$ then $c_{\kappa,S}\int_{S_{d-1}%
}pqw_{\kappa}d\omega=0$;\newline2) if $m=n$ then%
\begin{align*}
c_{\kappa,S}\int_{S_{d-1}}pqw_{\kappa}d\omega &  =\frac{c_{\kappa}\left(
2\pi\right)  ^{-d/2}}{2^{n}\left(  \gamma_{\kappa}+\frac{d}{2}\right)  _{n}%
}\int_{\mathbb{R}^{d}}p\left(  x\right)  q\left(  x\right)  w_{\kappa}\left(
x\right)  e^{-\left\vert x\right\vert ^{2}/2}dx\\
&  =\frac{1}{2^{n}\left(  \gamma_{\kappa}+\frac{d}{2}\right)  _{n}%
}\left\langle p,q\right\rangle _{\kappa}.
\end{align*}

\end{proposition}

That is, the spaces $\left\{  \mathcal{H}_{\kappa,n}:n\in\mathbb{N}%
_{0}\right\}  $ are pairwise orthogonal in $L^{2}\left(  S_{d-1},w_{\kappa
}d\omega\right)  $. By Proposition \ref{c4harmex} each polynomial agrees with
a harmonic one on $S_{d-1}$ and so $L^{2}\left(  S_{d-1},w_{\kappa}%
d\omega\right)  =\sum_{n=0}^{\infty}\oplus\mathcal{H}_{\kappa,n}$ by the
density of polynomials. In the next section we consider the reproducing and
Poisson kernels.

By a version of Hamburger's theorem $\Pi^{d}$ is dense in $L^{2}\left(
\mathbb{R}^{d},w_{\kappa}\left(  x\right)  e^{-\left\vert x\right\vert ^{2}%
/2}dx\right)  $. There exist orthogonal bases consisting of products of
harmonic polynomials and Laguerre polynomials with argument $\left\vert
x\right\vert ^{2}/2$. In general there is no explicit orthogonal basis for
$\mathcal{H}_{\kappa,n}$.

\begin{definition}
For $n\in\mathbb{N}_{0}$ let $X_{n}:=\mathrm{span}\left\{  p\left(  x\right)
L_{m}^{\left(  \alpha_{n}\right)  }\left(  \frac{\left\vert x\right\vert ^{2}%
}{2}\right)  :p\in\mathcal{H}_{\kappa,n},m\in\mathbb{N}_{0}\right\}  $, where
$\alpha_{n}=\gamma_{\kappa}+\frac{d}{2}+n-1$.
\end{definition}

Suppose $k,l,m,n\in\mathbb{N}_{0}$ and $p\in\mathcal{H}_{\kappa,n}%
,q\in\mathcal{H}_{\kappa,l}$ then $\left\langle \left\vert x\right\vert
^{2m}p\left(  x\right)  ,\left\vert x\right\vert ^{2k}q\left(  x\right)
\right\rangle _{\kappa}=0$ unless $n=l$ and $m=k$; if $2m+n\neq2k+l$ this
follows from part (1) of Theorem \ref{c4kform}; Proposition \ref{c4hmperp}
applies if $2m+n=2k+l$ and $m\neq k$. By equation \ref{LaguerreF} with
$s=-\frac{1}{2}$ $\left\langle L_{m}^{\left(  \alpha_{n}\right)  }\left(
\left\vert x\right\vert ^{2}/2\right)  p\left(  x\right)  ,L_{k}^{\left(
\alpha_{l}\right)  }\left(  \left\vert x\right\vert ^{2}/2\right)  q\left(
x\right)  \right\rangle _{g}=0$ unless $n=l$ and $m=k$. Thus%
\[
L^{2}\left(  \mathbb{R}^{d},w_{\kappa}\left(  x\right)  e^{-\left\vert
x\right\vert ^{2}/2}dx\right)  =\sum_{n=0}^{\infty}\oplus X_{n}.
\]
Suppose $p,q\in\mathcal{H}_{\kappa,n}$ and $k,m\in\mathbb{N}_{0}$ then%
\[
\left\langle L_{m}^{\left(  \alpha_{n}\right)  }\left(  \left\vert
x\right\vert ^{2}/2\right)  p\left(  x\right)  ,L_{k}^{\left(  \alpha
_{l}\right)  }\left(  \left\vert x\right\vert ^{2}/2\right)  q\left(
x\right)  \right\rangle _{g}=\delta_{mk}\frac{\left(  \alpha_{n}+1\right)
_{m}}{m!}\left\langle p,q\right\rangle _{g},
\]
and $\left\langle p,q\right\rangle _{g}=2^{n}\left(  \gamma_{\kappa}+\frac
{d}{2}\right)  _{n}c_{\kappa,S}\int_{S_{d-1}}pqw_{\kappa}d\omega$. These
formulae show how an orthogonal basis for $\mathcal{H}_{\kappa,n}$ can be used
to produce such a basis for $L^{2}\left(  \mathbb{R}^{d},w_{\kappa}\left(
x\right)  e^{-\left\vert x\right\vert ^{2}/2}dx\right)  $.

\section{The intertwining operator and the Dunkl kernel}

Several important objects can be defined when the form $\left\langle
\cdot,\cdot\right\rangle _{\kappa}$ is nondegenerate for specific numerical
values of $\kappa$. We refer to \textquotedblleft generic\textquotedblright%
\ $\kappa$ when $\kappa$ has some transcendental value (formal parameter), and
to \textquotedblleft specific" $\kappa$ when $\kappa$ takes on real values.
The form $\left\langle \cdot,\cdot\right\rangle _{\kappa}$ is defined for all
$\kappa$, and so is the following operator.

\begin{definition}
The operator $V_{\kappa}^{0}$ on $\Pi^{d}$ is given by%
\begin{align*}
V_{\kappa}^{0}p\left(  y\right)   &  :=\left\langle e^{\left\langle
y,x\right\rangle },p\left(  x\right)  \right\rangle _{\kappa},p\in\Pi^{d},\\
V_{\kappa}^{0}p\left(  y\right)   &  =\sum_{\alpha\in\mathbb{N}_{0}%
^{d},\left\vert \alpha\right\vert =n}\frac{1}{\alpha!}y^{\alpha}%
\mathcal{D}^{\alpha}p,p\in\Pi_{n}^{d}.
\end{align*}

\end{definition}

The second equation is the explicit effect of the formal operator defined in
the first equation. Note $e^{\left\langle y,x\right\rangle }=\sum
_{n=0}^{\infty}\frac{1}{n!}\sum_{\left\vert \alpha\right\vert =n}\binom
{n}{\alpha}x^{\alpha}y^{\alpha}$ (where $\alpha!=\prod_{i=1}^{d}\alpha_{i}!$
and $\binom{n}{\alpha}=\frac{n!}{\alpha!}$) and $\mathcal{D}^{\alpha
}p=\left\langle x^{\alpha},p\left(  x\right)  \right\rangle $ for $p\in\Pi
_{n}^{d},\left\vert \alpha\right\vert =n\in\mathbb{N}_{0}$. Also $V_{\kappa
}^{0}1=1$.

\begin{proposition}
If $1\leq i\leq d$ and $p\in\Pi^{d}$ then $\frac{\partial}{\partial x_{i}%
}V_{\kappa}^{0}p\left(  x\right)  =V_{\kappa}^{0}\mathcal{D}_{i}p\left(
x\right)  $. If $w\in G$ then $V_{\kappa}^{0}wp=wV_{\kappa}^{0}p$.
\end{proposition}

\begin{proof}
We have $\frac{\partial}{\partial y_{i}}V_{\kappa}^{0}p\left(  y\right)
=\left\langle x_{i}e^{\left\langle y,x\right\rangle },p\left(  x\right)
\right\rangle _{\kappa}=\left\langle e^{\left\langle y,x\right\rangle
},\mathcal{D}_{i}p\left(  x\right)  \right\rangle _{\kappa}$. For
$x,y\in\mathbb{R}^{d}$ let $f_{x}\left(  y\right)  =\left\langle
x,y\right\rangle $. Suppose $p\in\Pi_{n}^{d}$ then $V_{\kappa}^{0}\left(
wp\right)  \left(  x\right)  =\frac{1}{n!}\left\langle f_{x}^{n}%
,wp\right\rangle _{\kappa}=\frac{1}{n!}\left\langle \left(  w^{-1}%
f_{x}\right)  ^{n},p\right\rangle _{\kappa}$, and $w^{-1}f_{x}\left(
y\right)  =\left\langle x,wy\right\rangle =\left\langle w^{-1}x,y\right\rangle
$ so $V_{\kappa}^{0}\left(  wp\right)  \left(  x\right)  =V_{\kappa}%
^{0}p\left(  w^{-1}x\right)  $.
\end{proof}

\begin{definition}
For a specific $\kappa$ let $\mathrm{Rad}\left(  \kappa\right)  :=\left\{
p\in\Pi^{d}:\left\langle p,q\right\rangle _{\kappa}=0~\forall q\in\Pi
^{d}\right\}  $, the \emph{radical}.
\end{definition}

\begin{proposition}
The space $\mathrm{Rad}\left(  \kappa\right)  $ has the following
properties:\newline1) $p\in\mathrm{Rad}\left(  \kappa\right)  ,1\leq i\leq d,$
and $w\in G$ imply $x_{i}p\left(  x\right)  ,\mathcal{D}_{i}p\left(  x\right)
,wp\left(  x\right)  \in\mathrm{Rad}\left(  \kappa\right)  $,\newline2)
$\mathrm{Rad}\left(  \kappa\right)  =\ker V_{\kappa}^{0}$,\newline3)
$\mathrm{Rad}\left(  \kappa\right)  =\sum_{n=0}^{\infty}\left(  \mathrm{Rad}%
\left(  \kappa\right)  \cap\Pi_{n}^{d}\right)  $ (algebraic direct sum).
\end{proposition}

\begin{proof}
Part (1) follows directly from the properties of $\left\langle \cdot
,\cdot\right\rangle _{\kappa}$ and the definition of the radical. For part (2)
suppose the degree of $p$ is $n$ (that is, $p\in\sum_{j=0}^{n}\Pi_{j}^{d}$)
and $p\in\mathrm{Rad}\left(  \kappa\right)  $, then $V_{\kappa}^{0}p\left(
y\right)  =\sum_{j=0}^{n}\frac{1}{j!}\left\langle \left\langle
y,x\right\rangle ^{j},p\left(  x\right)  \right\rangle _{\kappa}=0$;
conversely suppose $V_{\kappa}^{0}p\left(  y\right)  =0$ then $\left\langle
x^{\alpha},p\left(  x\right)  \right\rangle _{\kappa}=0$ for all $\alpha
\in\mathbb{N}_{0}^{d}$ with $\left\vert \alpha\right\vert \leq n$, and thus
$p\in\mathrm{Rad}\left(  \kappa\right)  $. For part (3) the Euler operator
satisfies%
\[
\sum_{i=1}^{d}x_{i}\frac{\partial}{\partial x_{i}}=\sum_{i=1}^{d}%
x_{i}\mathcal{D}_{i}-\sum_{v\in R_{+}}\kappa\left(  v\right)  \left(
1-s_{v}\right)  .
\]
If $p\in\mathrm{Rad}\left(  \kappa\right)  $ then $\sum_{i=1}^{d}x_{i}%
\frac{\partial}{\partial x_{i}}p\left(  x\right)  \in\mathrm{Rad}\left(
\kappa\right)  $ by part (1), and hence $\mathrm{Rad}\left(  \kappa\right)  $
is the sum of its homogeneous subspaces.
\end{proof}

Part (1) implies that $\mathrm{Rad}\left(  \kappa\right)  $ is an ideal of the
\emph{rational Cherednik algebra} (an abstract algebra isomorphic to the
algebra of operators on $\Pi^{d}$ generated by the multipliers $x_{i}$, the
operators $\mathcal{D}_{i}$ and the group $G$). We can now set up the key
decomposition of the parameter space. Multiplicity functions can be identified
with points in $\mathbb{R}^{c}$ where $c$ is the number of $G$-orbits in $R$.

\begin{definition}
Let $\Lambda^{0}:=\left\{  \kappa:\mathrm{Rad}\left(  \kappa\right)
\neq\left\{  0\right\}  \right\}  $, the \emph{singular set}, and let
$\Lambda^{reg}:=\{\kappa:\mathrm{Rad}(\kappa)=\{0\}\}$, the \emph{regular set}.
\end{definition}

As a result of the papers of Opdam \cite{Op1} and Dunkl, de Jeu and Opdam
\cite{DJO} there is a concise description of the singular set for
indecomposable reflection groups: The value of the integral $\int
_{\mathbb{R}^{d}}w_{\kappa}\left(  x\right)  e^{-\left\vert x\right\vert
^{2}/2}dx$ is a meromorphic function of $\kappa$; the integral is defined for
$\kappa\geq0$ but the value extends analytically to $\mathbb{C}$. The poles
coincide with the singular set. For the one-class type the singular set is
$\left\{  -j/n_{m}:j\in\mathbb{N}_{0},1\leq m\leq d,j/n_{m}\notin
\mathbb{Z}\right\}  $, where the rank of $G$ is $d$ and the fundamental
degrees are $n_{1},\ldots,n_{d}$. The realization of the Gaussian form as an
integral shows that $\kappa\geq0$ implies $\kappa\in\Lambda^{reg}$. In the
next paragraphs we use superscripts $\left(  x\right)  ,\left(  y\right)  $ to
indicate the variable on which an operator acts.

\begin{definition}
For $\kappa\in\Lambda^{reg}$ let $V_{\kappa}:=\left(  V_{\kappa}^{0}\right)
^{-1}$, the \emph{intertwining operator}, and let $K_{\kappa,n}\left(
x,y\right)  :=\frac{1}{n!}V_{\kappa}^{\left(  y\right)  }\left\langle
x,y\right\rangle ^{n}=\sum\limits_{\alpha\in\mathbb{N}_{0}^{d},\left\vert
\alpha\right\vert =n}\frac{1}{\alpha!}x^{\alpha}V_{\kappa}\left(  y^{\alpha
}\right)  ,x,y\in\mathbb{R}^{d},n\in\mathbb{N}_{0}$. The polynomial
$K_{\kappa,n}$ is homogeneous of degree $n$ in both $x$ and $y$.
\end{definition}

\begin{theorem}
\label{c4KVprop}$K_{\kappa,n}$ and $V_{\kappa}$ have the following
properties:\newline1) if $p\in\Pi^{d}$ and $1\leq i\leq d$ then $\mathcal{D}%
_{i}\left(  V_{\kappa}p\right)  \left(  x\right)  =V_{\kappa}\left(
\frac{\partial}{\partial x_{i}}p\left(  x\right)  \right)  $; if $w\in G$ then
$wV_{\kappa}=V_{\kappa}w$;\newline2) $V_{\kappa}$ maps $\Pi_{n}^{d}$
one-to-one onto $\Pi_{n}^{d}$ for each $n$;\newline3) $\mathcal{D}%
_{i}^{\left(  y\right)  }K_{\kappa,n}\left(  x,y\right)  =x_{i}K_{\kappa
,n-1}\left(  x,y\right)  ;$\newline4) $\left\langle K_{\kappa,n}\left(
x,\cdot\right)  ,p\right\rangle _{\kappa}=p\left(  x\right)  $ for $p\in
\Pi_{n}^{d}$;\newline5) $K_{\kappa,n}\left(  x,y\right)  =K_{\kappa,n}\left(
y,x\right)  $ for all $x,y\in\mathbb{R}^{d}$ , and $K_{\kappa,n}\left(
wx,wy\right)  =K_{\kappa,n}\left(  x,y\right)  $ for each $w\in G$.
\end{theorem}

\begin{proof}
Parts (1) and (3) are straightforward. Part (2) holds because $V_{\kappa}^{0}$
maps $\Pi_{n}^{d}$ into $\Pi_{n}^{d}$ and its inverse exists. For part (4) let
$\partial_{y}^{\alpha}=\prod_{i=1}^{d}\left(  \frac{\partial}{\partial y_{i}%
}\right)  ^{\alpha_{i}}$; if $p\in\Pi_{n}^{d}$ then $p\left(  x\right)
=\sum_{\alpha\in\mathbb{N}_{0}^{d},\left\vert \alpha\right\vert =n}\frac
{1}{\alpha!}x^{\alpha}p\left(  \partial_{y}\right)  y^{\alpha}$. Apply
$V_{\kappa}^{\left(  y\right)  }$ to both sides (and the left side is
independent of $y)$ thus%
\begin{align*}
p\left(  x\right)   &  =V_{\kappa}^{\left(  y\right)  }p\left(  x\right)
=\sum_{\alpha\in\mathbb{N}_{0}^{d},\left\vert \alpha\right\vert =n}\frac
{1}{\alpha!}x^{\alpha}V_{\kappa}^{\left(  y\right)  }p\left(  \partial
_{y}\right)  y^{\alpha}\\
&  =\sum_{\alpha\in\mathbb{N}_{0}^{d},\left\vert \alpha\right\vert =n}\frac
{1}{\alpha!}x^{\alpha}p\left(  \mathcal{D}^{\left(  y\right)  }\right)
V_{\kappa}^{\left(  y\right)  }y^{\alpha}=\left\langle p,K_{\kappa,n}\left(
x,\cdot\right)  \right\rangle _{\kappa};
\end{align*}
and the form $\left\langle \cdot,\cdot\right\rangle _{\kappa}$ is symmetric.
For any $p,q\in\Pi_{n}^{d}$, by part (4),%
\begin{align*}
\left\langle p,q\right\rangle _{\kappa}  &  =\left\langle K_{\kappa,n}\left(
\cdot,\mathcal{D}^{\left(  y\right)  }\right)  p\left(  y\right)
,q\right\rangle _{\kappa}=K_{\kappa,n}\left(  \mathcal{D}^{\left(  x\right)
},\mathcal{D}^{\left(  y\right)  }\right)  p\left(  y\right)  q\left(
x\right) \\
&  =\left\langle q,p\right\rangle _{\kappa}.
\end{align*}
This implies $K_{\kappa,n}\left(  x,y\right)  =K_{\kappa,n}\left(  y,x\right)
$, because $\kappa\in\Lambda^{reg}$. Let $f_{x}\left(  y\right)  =K_{\kappa
,n}\left(  x,y\right)  ,w\in G$ and $p\in\Pi_{n}^{d}$, then $p\left(
x\right)  =\left\langle f_{x},p\right\rangle _{\kappa}$ and%
\begin{align*}
wp\left(  x\right)   &  =p\left(  w^{-1}x\right)  =\left\langle f_{w^{-1}%
x},p\right\rangle _{\kappa}\\
&  =\left\langle f_{x},wp\right\rangle _{\kappa}=\left\langle w^{-1}%
f_{x},p\right\rangle _{\kappa};
\end{align*}
thus $K_{\kappa,n}\left(  x,wy\right)  =w^{-1}f_{x}=f_{w^{-1}x}=K_{\kappa
,n}\left(  w^{-1}x,y\right)  $.
\end{proof}

Denote the formal sum $\sum_{n=0}^{\infty}K_{\kappa,n}\left(  x,y\right)  $ by
$K_{\kappa}\left(  x,y\right)  $. The question now arises: does the series
converge in a useful way? There are some strong results for $\kappa\geq0$. For
$x,y\in\mathbb{R}^{d}$ let $\rho\left(  x,y\right)  =\max_{w\in G}\left\vert
\left\langle x,wy\right\rangle \right\vert $.

\begin{theorem}
\label{c4Knbound}Suppose $\kappa\geq0$ and $x,y\in\mathbb{R}^{d}$, then
$\left\vert K_{\kappa,n}\left(  x,y\right)  \right\vert \leq\frac{1}{n!}%
\rho\left(  x,y\right)  ^{n}$ for all $n\in\mathbb{N}_{0}$, the series for
$K_{\kappa}$ converges uniformly and absolutely on compact subsets of
$\mathbb{R}^{d}\times\mathbb{R}^{d}$, and $\left\vert K_{\kappa}\left(
x,y\right)  \right\vert \leq e^{\rho\left(  x,y\right)  }$.
\end{theorem}

\begin{theorem}
\label{c4Vxpos}Suppose $\kappa\geq0$ and $x\in\mathbb{R}^{d}$ then there
exists a Baire probability measure $\mu_{x}$ with $spt\left(  \mu_{x}\right)
\subset co\left\{  wx:w\in G\right\}  $ (the convex hull) such that
$V_{\kappa}p\left(  x\right)  =\int_{\mathbb{R}^{d}}pd\mu_{x}$ for each
$p\in\Pi^{d}$.
\end{theorem}

\begin{corollary}
If $x,y\in\mathbb{R}^{d}$ then $K_{\kappa}\left(  x,y\right)  >0$ and
$\left\vert K_{\kappa}\left(  x,\mathrm{i}y\right)  \right\vert \leq1$.
\end{corollary}

\begin{proof}
By Fubini's theorem summation and integration can be interchanged in
$K_{\kappa}\left(  x,y\right)  =\sum\limits_{n=0}^{\infty}\frac{1}%
{n!}V_{\kappa}^{\left(  y\right)  }\left\langle x,y\right\rangle ^{n}%
=\int_{\mathbb{R}^{d}}e^{\left\langle x,z\right\rangle }d\mu_{y}\left(
z\right)  $. By homogeneity $K_{\kappa}(x,\mathrm{i}y)=\sum_{n=0}^{\infty
}\mathrm{i}^{n}K_{\kappa,n}(x,y)=\int_{\mathbb{R}^{d}}e^{\mathrm{i}\langle
x,z\rangle}\allowbreak\,d\mu_{y}\left(  z\right)  $.
\end{proof}

Theorem \ref{c4Knbound} was shown in Dunkl's paper \cite{Du2} constructing
$K_{\kappa}$, which is now called the \emph{Dunkl kernel}. Later R\"{o}sler
\cite{Ro2} proved Theorem \ref{c4Vxpos}. The inequality $\left\vert K_{\kappa
}\left(  x,\mathrm{i}y\right)  \right\vert \leq1$ will be used in the later
section on the Dunkl transform. There is a mean-value type result for
$V_{\kappa}$:

\begin{proposition}
Suppose $p\in\Pi^{d}$ then%
\[
\frac{c_{\kappa}}{\left(  2\pi\right)  ^{d/2}}\int_{\mathbb{R}^{d}}V_{\kappa
}p\left(  x\right)  w_{\kappa}\left(  x\right)  e^{-\left\vert x\right\vert
^{2}/2}dx=\frac{1}{\left(  2\pi\right)  ^{d/2}}\int_{\mathbb{R}^{d}}p\left(
x\right)  e^{-\left\vert x\right\vert ^{2}/2}dx.
\]

\end{proposition}

\begin{proof}
By Theorem \ref{c4Gint} the left side equals $\left\langle e^{\Delta_{\kappa
}/2}V_{\kappa}p,e^{\Delta_{\kappa}/2}1\right\rangle _{\kappa}=\left\langle
V_{\kappa}e^{\Delta/2}p,1\right\rangle _{\kappa}=\left\langle e^{\Delta
/2}p,1\right\rangle _{0}$ which equals the right side (the subscript $0$
indicates $\kappa=0)$.
\end{proof}

\begin{corollary}
If $f\in C\left(  \left\{  x\in\mathbb{R}^{d}:\left\vert x\right\vert
\leq1\right\}  \right)  $ and $\kappa>0$ then%
\[
c_{\kappa,S}\int_{S_{d-1}}\left(  V_{\kappa}f\right)  w_{\kappa}d\omega
=2\frac{\Gamma\left(  \gamma_{\kappa}+\frac{d}{2}\right)  }{\Gamma\left(
\gamma_{\kappa}\right)  \Gamma\left(  \frac{d}{2}\right)  }\int_{\left\vert
x\right\vert \leq1}f\left(  x\right)  \left(  1-\left\vert x\right\vert
^{2}\right)  ^{\gamma_{\kappa}-1}dx.
\]

\end{corollary}

This is proven by applying the Proposition to a homogeneous polynomial and
then evaluating the two integrals in spherical polar coordinates (see
Proposition \ref{c4sphint}; only the even degree case need be computed). The
formula extends to continuous functions by Theorem \ref{c4Vxpos}. This result
is due to Xu \cite{Xu1}.

Opdam \cite{Op1} defined a Bessel function in the multiplicity function
context. His approach was through $G$-invariant differential operators
commuting with $\Delta+2\sum_{v\in R_{+}}\kappa\left(  v\right)
\frac{\left\langle v,\nabla\right\rangle }{\left\langle x,v\right\rangle }$
(the differential part of $\Delta_{\kappa}$). The result is that $J_{G}\left(
x,y\right)  =\frac{1}{\#G}\sum_{w\in G}K_{\kappa}\left(  wx,y\right)  $ is
real-entire in $x,y$ for each $\kappa\in\Lambda^{reg}$, and $J_{G}\left(
x,y\right)  $ is meromorphic in $\kappa$ with poles on $\Lambda^{0}$. In the
paper \cite[Rem. 6.12]{Op1} Opdam observed that $J_{G}$ can be interpreted as
a spherical function on a Euclidean symmetric space, when $G$ is a Weyl group
and $\kappa$ takes values in certain discrete sets.

The properties of $K_{\kappa,n}$ described in Theorem \ref{c4KVprop} extend to
$K_{\kappa}$:\newline1) $\mathcal{D}_{i}^{\left(  y\right)  }K_{\kappa}\left(
x,y\right)  =x_{i}K_{\kappa}\left(  x,y\right)  $ for $1\leq i\leq d$%
,\newline2) $\left\langle K_{\kappa}\left(  x,\cdot\right)  ,p\right\rangle
_{\kappa}=p\left(  x\right)  $ for $p\in\Pi^{d}$;\newline3) $K_{\kappa}\left(
x,y\right)  =K_{\kappa}\left(  y,x\right)  $ for all $x,y\in\mathbb{R}^{d}$ ,
and $K_{\kappa}\left(  wx,wy\right)  =K_{\kappa}\left(  x,y\right)  $ for each
$w\in G$.

Property (3) shows that $J_{G}\left(  wx,y\right)  =J_{G}\left(  x,wy\right)
=J_{G}\left(  x,y\right)  $ for all $w\in G$.

\subsection{Example: $Z_{2}$}

The objects described above can be stated explicitly for the smallest
reflection group. We use $\kappa$ for the value of the multiplicity function
and suppress the subscript \textquotedblleft1\textquotedblright\ $($for
example, in $x_{1}$). Throughout $n\in\mathbb{N}_{0}$.\newline1)
$\mathcal{D}p\left(  x\right)  =\partial_{x}p\left(  x\right)  +\kappa
\frac{p\left(  x\right)  -p\left(  -x\right)  }{x}$,\newline2) $\left\langle
x^{2n},x^{2n}\right\rangle _{\kappa}=2^{2n}n!\left(  \kappa+\frac{1}%
{2}\right)  _{n}$, $\left\langle x^{2n+1},x^{2n+1}\right\rangle _{\kappa
}=2^{2n+1}n!\left(  \kappa+\frac{1}{2}\right)  _{n+1}$;\newline3) $V_{\kappa
}^{0}x^{2n}=\frac{\left(  \kappa+\frac{1}{2}\right)  _{n}}{\left(  \frac{1}%
{2}\right)  _{n}}x^{2n}$, $V_{\kappa}^{0}x^{2n+1}=\frac{\left(  \kappa
+\frac{1}{2}\right)  _{n+1}}{\left(  \frac{1}{2}\right)  _{n+1}}x^{2n+1}%
$;\newline4) $\Lambda^{0}=\left\{  -\frac{1}{2},-\frac{3}{2},\ldots\right\}
$; if $\kappa=-\frac{1}{2}-n$ then $\mathrm{Rad}\left(  \kappa\right)
=\mathrm{span}\left\{  x^{m}:m\geq2n+1\right\}  ;$\newline5) $\left\langle
p,q\right\rangle _{g}=2^{-\kappa-1/2}\Gamma\left(  \kappa+\frac{1}{2}\right)
^{-1}\int_{-\infty}^{\infty}p\left(  x\right)  q\left(  x\right)  \left\vert
x\right\vert ^{2\kappa}e^{-x^{2}/2}dx$, valid for $\kappa>-\frac{1}{2}%
$;\newline6) $e^{-\mathcal{D}^{2}/2}x^{2n}=\left(  -1\right)  ^{m}n!2^{n}%
L_{n}^{\left(  \kappa-1/2\right)  }\left(  \frac{x^{2}}{2}\right)  $,
$e^{-\mathcal{D}^{2}/2}x^{2n+1}=\left(  -1\right)  ^{n}n!2^{n}xL_{n}^{\left(
\kappa+1/2\right)  }\left(  \frac{x^{2}}{2}\right)  $;\newline7) $V_{\kappa
}p\left(  x\right)  =\frac{\Gamma\left(  \frac{1}{2}\right)  \Gamma\left(
\kappa\right)  }{\Gamma\left(  \kappa+\frac{1}{2}\right)  }\int_{-1}%
^{1}p\left(  xt\right)  \left(  1+t\right)  ^{\kappa}\left(  1-t\right)
^{\kappa-1}dt$, for $\kappa>0$;\newline8) $K_{\kappa}\left(  x,y\right)
=\sum\limits_{n=0}^{\infty}\frac{1}{n!\left(  \kappa+\frac{1}{2}\right)  _{n}%
}\left(  \frac{xy}{2}\right)  ^{2n}+\frac{xy}{1+2\kappa}\sum\limits_{n=0}%
^{\infty}\frac{1}{\left(  \kappa+\frac{3}{2}\right)  _{n}~n!}\left(  \frac
{xy}{2}\right)  ^{2n}$.

Part (7) can be shown by substituting $p\left(  x\right)  =x^{n}$ in the
integral and using part (3). In part (8) note that the modified Bessel
function $I_{\kappa-1/2}\left(  x\right)  =\frac{\left(  x/2\right)
^{\kappa-1/2}}{\Gamma\left(  \kappa+1/2\right)  }\sum\limits_{n=0}^{\infty
}\frac{1}{n!\left(  \kappa+\frac{1}{2}\right)  _{n}}\left(  \frac{x}%
{2}\right)  ^{2n}$. This partly explains the use of \textquotedblleft
Bessel\textquotedblright\ in naming $J_{G}\left(  x,y\right)  $.

\subsection{Asymptotic properties of the Dunkl kernel}

De Jeu and R\"{o}sler \cite{dJR} proved the following results concerning the
limiting behavior of $K_{\kappa}\left(  x,y\right)  $ as $x\rightarrow\infty$,
for $\kappa\geq0$. The variable $x$ is restricted to the inside of a chamber.
The \emph{fundamental chamber} corresponds to the positive root system $R_{+}$
, also let $\delta>0$, then define%
\begin{align*}
\mathcal{C}  &  :=\left\{  x\in\mathbb{R}^{d}:\left\langle x,v\right\rangle
>0~\forall v\in R_{+}\right\}  ,\\
\mathcal{C}_{\delta}  &  :=\left\{  x\in\mathbb{R}^{d}:\left\langle
x,v\right\rangle >\delta\left\vert x\right\vert ~\forall v\in R_{+}\right\}  .
\end{align*}
The walls of $\mathcal{C}$ are the hyperplanes $v^{\bot}$ for the simple roots
$v$.

\begin{theorem}
For each $w\in G$ there is a constant $A_{w}$ such that for all $y\in
\mathcal{C}$%
\[
\lim_{x\in\mathcal{C}_{\delta},\left\vert x\right\vert \rightarrow\infty}%
\sqrt{w_{\kappa}\left(  x\right)  w_{\kappa}\left(  y\right)  }e^{-\mathrm{i}%
\left\langle x,wy\right\rangle }K_{\kappa}\left(  \mathrm{i}x,wy\right)
=A_{w}.
\]

\end{theorem}

Recall $\mathbb{\gamma}_{\kappa}=\sum_{v\in R_{+}}\kappa\left(  v\right)  $.
For $z\in\mathbb{C}$ with $\operatorname{Re}z\geq0$ let $z^{\mathbb{\gamma
}_{\kappa}}$ denote the principal branch ($1^{\mathbb{\gamma}_{\kappa}}=1$).
(See Definition \ref{c4Gauss2} for $c_{\kappa}$.)

\begin{theorem}
\label{c4asymp}The constant $A_{1}=\left(  \mathrm{i}^{\mathbb{\gamma}%
_{\kappa}}c_{\kappa}\right)  ^{-1}$ and for $x,y\in\mathcal{C}$%
\[
\lim_{\operatorname{Re}z\geq0,z\rightarrow\infty}z^{\mathbb{\gamma}_{\kappa}%
}e^{-z\left\langle x,y\right\rangle }K_{\kappa}\left(  zx,y\right)  =\frac
{1}{c_{\kappa}\sqrt{w_{\kappa}\left(  x\right)  w_{\kappa}\left(  y\right)  }%
}.
\]

\end{theorem}

This limit is used in the context of a heat kernel.

\subsection{The heat kernel}

For functions defined on $\mathbb{R}^{d}\times\left(  0,\infty\right)  $ the
\emph{generalized heat equation} is%
\[
\Delta_{\kappa}u\left(  x,t\right)  -\frac{\partial}{\partial t}u\left(
x,t\right)  =0.
\]
The associated boundary-value problem is to find the solution $u$ such that
$u\left(  x,0\right)  =f\left(  x\right)  $ where $f$ is a given bounded
continuous function on $\mathbb{R}^{d}$.

\begin{definition}
For $x,y\in\mathbb{R}^{d}$ and $t>0$ the \emph{generalized heat kernel}
$\Gamma_{\kappa}$ is given by%
\[
\Gamma_{\kappa}\left(  t,x,y\right)  :=\frac{c_{\kappa}}{\left(  2t\right)
^{\mathbb{\gamma}_{\kappa}+d/2}\left(  2\pi\right)  ^{d/2}}\exp\left(
-\frac{\left\vert x\right\vert ^{2}+\left\vert y\right\vert ^{2}}{4t}\right)
K_{\kappa}\left(  \frac{x}{\sqrt{2t}},\frac{y}{\sqrt{2t}}\right)  .
\]

\end{definition}

\begin{definition}
For a bounded continuous function $f$ on $\mathbb{R}^{d}$ and $t>0$ let%
\[
H\left(  t\right)  f\left(  x\right)  :=\int_{\mathbb{R}^{d}}f\left(
y\right)  \Gamma_{\kappa}\left(  t,x,y\right)  w_{\kappa}\left(  y\right)
dy.
\]

\end{definition}

\begin{theorem}
Suppose $f\in\mathcal{S}\left(  \mathbb{R}^{d}\right)  $ (the Schwartz space)
then $H\left(  t\right)  f\in\mathcal{S}\left(  \mathbb{R}^{d}\right)  $ for
all $t>0$, $H\left(  s\right)  H\left(  t\right)  f=H\left(  s+t\right)  f$
for all $s,t>0$, and $\lim\limits_{t\rightarrow0_{+}}\sup\limits_{x}\left\vert
H\left(  t\right)  f\left(  x\right)  -f\left(  x\right)  \right\vert =0$.
Furthermore the function $u\left(  x,t\right)  =H\left(  t\right)  f\left(
x\right)  $ for $t>0$, $=f\left(  x\right)  $ for $t=0$, solves the
boundary-value problem.
\end{theorem}

These results are due to R\"{o}sler \cite{Ro1}. Theorem \ref{c4asymp} implies%
\[
\lim_{t\rightarrow0_{+}}\frac{\sqrt{w_{\kappa}\left(  x\right)  w_{\kappa
}\left(  y\right)  }\Gamma_{\kappa}\left(  t,x,y\right)  }{\Gamma_{0}\left(
t,x,y\right)  }=1
\]
for all $x,y\in\mathcal{C}$.

There is an associated c\`{a}dl\`{a}g Markov process $X=\left(  X_{t}\right)
_{t\geq0}$ with infinitesimal generator $\frac{1}{2}\Delta_{\kappa}$. The
semigroup densities are
\[
p_{t}^{\left(  \kappa\right)  }\left(  x,y\right)  :=\Gamma_{\kappa}\left(
\frac{t}{2},x,y\right)  w_{\kappa}\left(  y\right)  .
\]
For further details see R\"{o}sler and Voit \cite{RoV}, Gallardo and Yor
\cite{GY}, and the monograph \cite{GRY}.

\section{The Dunkl transform}

The Dunkl kernel is used to define a generalization of the Fourier transform.
The Fourier integral kernel $e^{-\mathrm{i}\left\langle x,y\right\rangle }$ is
replaced by $K_{\kappa}\left(  x,-\mathrm{i}y\right)  w_{\kappa}\left(
x\right)  $. Throughout this section $\kappa\geq0$. Recall $\left\vert
K_{\kappa}\left(  x,-\mathrm{i}y\right)  \right\vert \leq1$ for all
$x,y\in\mathbb{R}^{d}$.

\begin{definition}
For $f\in L^{1}\left(  \mathbb{R}^{d},w_{\kappa}\left(  x\right)  dx\right)  $
let%
\[
\mathcal{F}f\left(  y\right)  :=\frac{c_{\kappa}}{\left(  2\pi\right)  ^{d/2}%
}\int_{\mathbb{R}^{d}}f\left(  x\right)  K_{\kappa}\left(  x,-\mathrm{i}%
y\right)  w_{\kappa}\left(  x\right)  dx.
\]

\end{definition}

By finding a set of eigenfunctions of $\mathcal{F}$ which is dense in
$L^{2}\left(  \mathbb{R}^{d},w_{\kappa}\left(  x\right)  dx\right)  $ (by
Hamburger's theorem) we show that $\mathcal{F}$ is an $L^{2}$-isometry, has
period $4$, and $\mathcal{F}x_{j}=\mathrm{i}\mathcal{D}_{j}\mathcal{F}$ for
$1\leq j\leq d$. Convergence arguments, mostly depending on the dominated
convergence theorem, are omitted, and appropriate smoothness restrictions on
functions are implicitly assumed.

\begin{theorem}
Let $f\left(  x\right)  =p\left(  x\right)  L_{m}^{\left(  \alpha\right)
}\left(  \left\vert x\right\vert ^{2}\right)  e^{-\left\vert x\right\vert
^{2}/2}$ where $m,n\in\mathbb{N}_{0}$, $\alpha=n+\frac{d}{2}+\mathbb{\gamma
}_{\kappa}-1$ and $p\in\mathcal{H}_{\kappa,n}$, then $\mathcal{F}f\left(
y\right)  =\left(  -\mathrm{i}\right)  ^{n+2m}f\left(  y\right)
,y\in\mathbb{R}^{d}$.
\end{theorem}

\begin{proof}
Suppose $q$ is an arbitrary polynomial of degree $n$, for some $n$, and $N\geq
n$ then the formula $\left\langle q,\sum_{j=0}^{N}K_{\kappa,j}\left(
\cdot,u\right)  \right\rangle _{\kappa}=q\left(  u\right)  $ is valid for all
$u\in\mathbb{C}^{d}$, since it is a polynomial relation. By Theorem
\ref{c4Gint} and letting $N\rightarrow\infty$%
\[
\frac{c_{\kappa}}{\left(  2\pi\right)  ^{d/2}}\int_{\mathbb{R}^{d}}%
e^{-\Delta_{\kappa}/2}q\left(  x\right)  e^{-\Delta_{\kappa}^{\left(
x\right)  }/2}K_{\kappa}\left(  x,u\right)  e^{-\left\vert x\right\vert
^{2}/2}w_{\kappa}\left(  x\right)  dx=q\left(  u\right)  ,
\]
and $e^{-\Delta_{\kappa}^{\left(  x\right)  }/2}K_{\kappa}\left(  x,u\right)
=\exp\left(  -\frac{1}{2}\sum_{j=1}^{d}u_{j}^{2}\right)  K_{\kappa}\left(
x,u\right)  $; thus%
\[
\frac{c_{\kappa}}{\left(  2\pi\right)  ^{d/2}}\int_{\mathbb{R}^{d}}%
e^{-\Delta_{\kappa}/2}q\left(  x\right)  K_{\kappa}\left(  x,u\right)
e^{-\left\vert x\right\vert ^{2}/2}w_{\kappa}\left(  x\right)  dx=\exp\left(
\frac{1}{2}\sum_{j=1}^{d}u_{j}^{2}\right)  q\left(  u\right)  .
\]
Set $q\left(  x\right)  =e^{\Delta_{\kappa}/4}\left(  \left\vert x\right\vert
^{2m}p\left(  x\right)  \right)  =m!L_{m}^{\left(  \alpha\right)  }\left(
-\sum_{j=1}^{d}x_{j}^{2}\right)  p\left(  x\right)  $ (by equation
\ref{LaguerreF}), and set $u=-\mathrm{i}y$ in the previous equation. In the
left side we have $e^{-\Delta_{\kappa}/4}\left(  \left\vert x\right\vert
^{2m}p\left(  x\right)  \right)  e^{-\left\vert x\right\vert ^{2}/2}=\left(
-1\right)  ^{m}m!f\left(  x\right)  $; and the right side equals
$m!e^{-\left\vert y\right\vert ^{2}/2}L_{m}^{\left(  \alpha\right)  }\left(
\left\vert y^{2}\right\vert \right)  p\left(  -\mathrm{i}y\right)  =\left(
-\mathrm{i}\right)  ^{n}m!f\left(  y\right)  $.
\end{proof}

\begin{corollary}
If $f\in L^{2}\left(  \mathbb{R}^{d},w_{\kappa}\left(  x\right)  dx\right)  $
then $\int_{\mathbb{R}^{d}}\left\vert \mathcal{F}f\left(  y\right)
\right\vert ^{2}w_{\kappa}\left(  y\right)  dy=\int_{\mathbb{R}^{d}}\left\vert
f\left(  y\right)  \right\vert ^{2}w_{\kappa}\left(  x\right)  dx$, and
$\mathcal{F}^{2}f\left(  x\right)  =f\left(  -x\right)  $ for almost all
$x\in\mathbb{R}^{d}$.
\end{corollary}

Suppose $f\left(  x\right)  ,\left\vert x\right\vert f\left(  x\right)  \in
L^{1}\left(  \mathbb{R}^{d},w_{\kappa}\left(  x\right)  dx\right)  $ and
$1\leq j\leq d$ then%
\[
\mathcal{D}_{j}\mathcal{F}f\left(  y\right)  =-\mathrm{i}\mathcal{F}\left(
x_{j}f\left(  x\right)  \right)  \left(  y\right)  ,y\in\mathbb{R}^{d},
\]
since $K_{\kappa}\left(  x,-\mathrm{i}y\right)  =K_{\kappa}\left(
-\mathrm{i}x,y\right)  $. The $G$-invariance property of $K_{\kappa}$ implies
that $w\mathcal{F=F}w$ for all $w\in G$.

The transform and the $L^{2}$-isometry result first appeared in \cite{Du3}.
The uniform boundedness of $\left\vert K_{\kappa}\left(  x,-\mathrm{i}%
y\right)  \right\vert $ was shown by de Jeu \cite{dJ1}. For the special case
$d=1,G=Z_{2}$ and even functions $f$ on $\mathbb{R}$ the transform
$\mathcal{F}$ essentially coincides with the classical Hankel transform.

\section{The Poisson kernel}

In this section $\kappa\geq0$. There is a natural boundary value problem for
harmonic functions. Given $f\in C\left(  S_{d-1}\right)  $ find the function
$P\left[  f\right]  $ which is smooth on $\left\{  x:\left\vert x\right\vert
<1\right\}  $, $\Delta_{\kappa}P\left[  f\right]  =0$ and $\lim_{r\rightarrow
1_{-}}P\left[  f\right]  \left(  rx\right)  =f\left(  x\right)  $ for $x\in
S_{d-1}$. We outline the argument for polynomial functions on $S_{d-1}$.

Let $n\in\mathbb{N}_{0}$ and $x,y\in\mathbb{R}^{d}$. By the definition of
$\pi_{\kappa,n}$ and Theorem \ref{c4KVprop}(3)%
\[
\pi_{\kappa,n}^{\left(  x\right)  }K_{\kappa,n}\left(  x,y\right)  =\sum
_{j=0}^{\left\lfloor n/2\right\rfloor }\frac{1}{4^{j}j!\left(  -\gamma
_{\kappa}-n+2-d/2\right)  _{j}}\left\vert x\right\vert ^{2j}\left\vert
y\right\vert ^{2j}K_{\kappa,n-2j}\left(  x,y\right)  ,
\]
Applying $\pi_{n,\kappa}^{\left(  x\right)  }$ to the reproducing equation
$\left\langle p,K_{\kappa,n}\left(  x,\cdot\right)  \right\rangle _{\kappa
}=p\left(  x\right)  $ for $p\in\Pi_{n}^{d}$ we obtain%
\begin{align*}
\pi_{\kappa,n}p\left(  x\right)   &  =\left\langle p,\pi_{\kappa,n}^{\left(
x\right)  }K_{\kappa,n}\left(  x,\cdot\right)  \right\rangle _{\kappa}%
=\frac{c_{\kappa}}{\left(  2\pi\right)  ^{d/2}}\int_{\mathbb{R}^{d}}%
e^{-\Delta_{\kappa}/2}p\left(  y\right)  \pi_{\kappa,n}^{\left(  x\right)
}K_{\kappa,n}\left(  x,y\right)  e^{-\left\vert x\right\vert ^{2}/2}w_{\kappa
}\left(  y\right)  dy\\
&  =2^{n}\left(  \gamma_{\kappa}+\frac{d}{2}\right)  _{n}c_{\kappa,S}%
\int_{S_{d-1}}p\left(  y\right)  \pi_{\kappa,n}^{\left(  x\right)  }%
K_{\kappa,n}\left(  x,y\right)  w_{\kappa}\left(  y\right)  d\omega\left(
y\right)  ,
\end{align*}
by Proposition \ref{c4sphint}, and because $e^{-\Delta_{\kappa}/2}p\left(
y\right)  =p\left(  y\right)  +p^{\prime}\left(  y\right)  $ where $p^{\prime
}$ is of degree $\leq n-2$ and is thus orthogonal to .$\pi_{\kappa,n}^{\left(
x\right)  }K_{\kappa,n}\left(  x,y\right)  $.

\begin{definition}
For $n\in$ $\mathbb{N}_{0}$ and $x,y\in\mathbb{R}^{d}$ let
\[
P_{\kappa,n}\left(  x,y\right)  :=2^{n}\left(  \gamma_{\kappa}+\frac{d}%
{2}\right)  _{n}\sum_{j=0}^{\left\lfloor n/2\right\rfloor }\frac{1}%
{4^{j}j!\left(  -\gamma_{\kappa}-n+2-d/2\right)  _{j}}\left\vert x\right\vert
^{2j}\left\vert y\right\vert ^{2j}K_{\kappa,n-2j}\left(  x,y\right)  .
\]

\end{definition}

By the decomposition \ref{c4harmex} any polynomial $p$ satisfies $p\left(
x\right)  =\sum_{n,k\geq0}\left\vert x\right\vert ^{2k}p_{n,k}\left(
x\right)  $ where each $p_{n,k}\in\mathcal{H}_{\kappa,n}$ , and so $p$ agrees
with the harmonic polynomial $q\left(  x\right)  =\sum_{n,k\geq0}%
p_{n,k}\left(  x\right)  $ on $S_{d-1}$ . Thus%
\[
q\left(  x\right)  =c_{\kappa,S}\int_{S_{d-1}}p\left(  y\right)  \sum
_{j=0}^{N}P_{\kappa,j}\left(  x,y\right)  w_{\kappa}\left(  y\right)
d\omega\left(  y\right)  ,
\]
where $N$ is sufficiently large. The series converges for $\left\vert
x\right\vert <1,\left\vert y\right\vert =1$.

\begin{theorem}
For $\left\vert x\right\vert <1,\left\vert y\right\vert =1$%
\[
\sum_{j=0}^{\infty}P_{\kappa,j}\left(  x,y\right)  =V_{\kappa}^{\left(
y\right)  }\left(  \frac{1-\left\vert x\right\vert ^{2}}{\left(
1-2\left\langle x,y\right\rangle +\left\vert x\right\vert ^{2}\right)
^{\mathbb{\gamma}_{\kappa}+d/2}}\right)  .
\]

\end{theorem}

The result follows from expanding the right hand side as a series in
$V_{\kappa}^{\left(  y\right)  }\left(  \left\langle x,y\right\rangle
^{j}\right)  $. The left side of the equation is thus the Poisson kernel for
harmonic functions in the unit ball. For fixed $x$ with $\left\vert
x\right\vert <1$ the denominator in the $V_{\kappa}^{\left(  y\right)  }$-term
does not vanish for $\left\vert y\right\vert \leq1$. There is a formula for
$P_{\kappa,n}\left(  x,y\right)  $ restricted to $\left\vert x\right\vert =1$:%
\begin{align*}
&  V_{\kappa}^{\left(  x\right)  }\frac{n+\alpha}{\alpha}C_{n}^{\alpha}\left(
\left\langle x,y/\left\vert y\right\vert \right\rangle \right)  \left\vert
y\right\vert ^{n}\\
&  =2^{n}\left(  \gamma_{\kappa}+\frac{d}{2}\right)  _{n}\sum_{j=0}%
^{\left\lfloor n/2\right\rfloor }\frac{1}{4^{j}j!\left(  -\gamma_{\kappa
}-n+2-d/2\right)  _{j}}\left\vert y\right\vert ^{2j}K_{\kappa,n-2j}\left(
x,y\right)  ,
\end{align*}
where $C_{n}^{\alpha}$ is the Gegenbauer polynomial of degree $n$ and index
$\alpha=\mathbb{\gamma}_{\kappa}+\frac{d}{2}-1$. (for the exceptional case
$\alpha=0$, replace the left side by $V_{\kappa}^{\left(  x\right)  }%
2T_{n}\left(  \left\langle x,y/\left\vert y\right\vert \right\rangle \right)
\left\vert y\right\vert ^{n}$ for $n\geq1$; $T_{n}$ is the Chebyshev
polynomial of the first kind). The formula is suggested by a generating
function for these polynomials. Maslouhi and Youssfi \cite{MY} studied the
properties of the Poisson kernel in connection with $L^{p}$-type convergence
and with a generalized translation.

\section{Harmonic polynomials for $\mathbb{R}^{2}$}

This section exhibits the classical Gegenbauer and Jacobi polynomials as
spherical harmonics for the groups $\mathbb{Z}_{2}$ and $\mathbb{Z}_{2}%
\times\mathbb{Z}_{2}$ with root systems $\left\{  \pm\varepsilon_{2}\right\}
$ and $\left\{  \pm\varepsilon_{1},\pm\varepsilon_{2}\right\}  $ respectively.
We mention here that the structure associated with the group $\mathbb{Z}%
_{2}^{d}$, $R=\left\{  \pm\varepsilon_{i}:1\leq i\leq d\right\}  $ and
$\kappa\left(  \varepsilon_{i}\right)  =\kappa_{i}$ can be analyzed in a
similar way and leads to multi-variable Jacobi polynomials orthogonal on a
simplex; see Chapter 2.

Polar coordinates will be used: $x_{1}=r\cos\theta,x_{2}=r\sin\theta$,
$r\geq0,-\pi\leq\theta\leq\pi$.

\subsection{One parameter}

Let $R=\left\{  \pm\varepsilon_{2}\right\}  $ and $\kappa>-\frac{1}{2}$,
then\newline1) $w_{\kappa}\left(  x\right)  =\left\vert x_{2}\right\vert
^{2\kappa}=r^{2\kappa}\left\vert \sin\theta\right\vert ^{2\kappa}$,\newline2)
$\mathcal{D}_{1}p\left(  x\right)  =\frac{\partial}{\partial x_{1}}p\left(
x\right)  ,\mathcal{D}_{2}p\left(  x\right)  =\frac{\partial}{\partial x_{2}%
}p\left(  x\right)  +\kappa\frac{p\left(  x\right)  -p\left(  x_{1}%
,-x_{2}\right)  }{x_{2}}$,\newline3) $\Delta_{\kappa}p\left(  x\right)
=\Delta p\left(  x\right)  +\frac{\kappa}{x_{2}}\left(  2\frac{\partial
}{\partial x_{2}}p\left(  x\right)  -\frac{p\left(  x\right)  -p\left(
x_{1},-x_{2}\right)  }{x_{2}}\right)  $.

For $n\geq1$ the space $\mathcal{H}_{\kappa,n}$ contains an orthogonal basis
consisting of two polynomials, $p_{n,0}$ being even and $p_{n,1}$ being odd in
$x_{2}$:%
\begin{align*}
p_{n,0}\left(  x\right)   &  :=r^{n}C_{n}^{\kappa}\left(  \cos\theta\right)
,\\
p_{n,1}\left(  x\right)   &  :=r^{n}\sin\theta~C_{n-1}^{\kappa+1}\left(
\cos\theta\right)  .
\end{align*}
Let $a_{\kappa}=\frac{\Gamma\left(  \kappa+1\right)  }{2\sqrt{\pi}%
\Gamma\left(  \kappa+1/2\right)  }$, the normalizing constant such that
$a_{\kappa}\int_{-\pi}^{\pi}\left\vert \sin\theta\right\vert ^{2\kappa}%
d\theta=1$; then%
\begin{align*}
a_{\kappa}\int_{-\pi}^{\pi}p_{n,0}\left(  \cos\theta,\sin\theta\right)
^{2}\left\vert \sin\theta\right\vert ^{2\kappa}d\theta &  =\frac{\kappa\left(
2\kappa\right)  _{n}}{\left(  n+\kappa\right)  n!},\\
a_{\kappa}\int_{-\pi}^{\pi}p_{n,1}\left(  \cos\theta,\sin\theta\right)
^{2}\left\vert \sin\theta\right\vert ^{2\kappa}d\theta &  =\frac{\left(
\kappa+\frac{1}{2}\right)  \left(  2\kappa+2\right)  _{n-1}}{\left(
n+\kappa\right)  \left(  n-1\right)  !}.
\end{align*}
If the Gegenbauer polynomials are replaced by $P_{n}^{\kappa}=\frac
{n!}{\left(  2\kappa\right)  _{n}}C_{n}^{\kappa}$ (similarly for
$C_{n-1}^{\kappa+1}$, these are normalized by $P_{n}^{\kappa}\left(  1\right)
=1$) then at $\kappa=0$ one obtains $p_{n,0}=r^{n}\cos n\theta$ and
$p_{n,1}=r^{n}\sin n\theta$.

\subsection{Two parameters}

Let $R=\left\{  \pm\varepsilon_{1,}\pm\varepsilon_{2}\right\}  $ and
$\kappa_{1},\kappa_{2}>-\frac{1}{2}$, then\newline1) $w_{\kappa}\left(
x\right)  =\left\vert x_{1}\right\vert ^{2\kappa_{1}}\left\vert x_{2}%
\right\vert ^{2\kappa_{2}}=r^{2\kappa_{1}+2\kappa_{2}}\left\vert \cos
\theta\right\vert ^{2\kappa_{1}}\left\vert \sin\theta\right\vert ^{2\kappa
_{2}}$,\newline2) $\mathcal{D}_{1}p\left(  x\right)  =\frac{\partial}{\partial
x_{1}}p\left(  x\right)  +\kappa_{1}\frac{p\left(  x\right)  -p\left(
-x_{1},x_{2}\right)  }{x_{2}},\mathcal{D}_{2}p\left(  x\right)  =\frac
{\partial}{\partial x_{2}}p\left(  x\right)  +\kappa_{2}\frac{p\left(
x\right)  -p\left(  x_{1},-x_{2}\right)  }{x_{2}}$,\newline3) $\Delta_{\kappa
}p\left(  x\right)  =\Delta p\left(  x\right)  +\frac{\kappa_{1}}{x_{1}%
}\left(  2\frac{\partial}{\partial x_{1}}p\left(  x\right)  -\frac{p\left(
x\right)  -p\left(  -x_{1},x_{2}\right)  }{x_{1}}\right)  $+$\frac{\kappa_{2}%
}{x_{2}}\left(  2\frac{\partial}{\partial x_{2}}p\left(  x\right)
-\frac{p\left(  x\right)  -p\left(  x_{1},-x_{2}\right)  }{x_{2}}\right)  .$

There are 4 families of harmonic polynomials. The first subscript indicates
the degree and the second subscript is used to indicate the parity type, for
example $01$ denotes \textquotedblleft even in $x_{1}$, odd in $x_{2}%
$.\textquotedblright%
\begin{align*}
p_{2n,00}\left(  x\right)   &  :=r^{2n}P_{n}^{\left(  \kappa_{2}%
-1/2,\kappa_{1}-1/2\right)  }\left(  \cos2\theta\right)  ,\\
p_{2n,11}\left(  x\right)   &  :=r^{2n}\sin2\theta~P_{n-1}^{\left(  \kappa
_{2}+1/2,\kappa_{1}+1/2\right)  }\left(  \cos2\theta\right)  ,\\
p_{2n+1,10}\left(  x\right)   &  :=r^{2n+1}\cos\theta~P_{n}^{\left(
\kappa_{2}-1/2,\kappa_{1}+1/2\right)  }\left(  \cos2\theta\right)  ,\\
p_{2n+1,01}\left(  x\right)   &  :=r^{2n+1}\sin\theta~P_{n}^{\left(
\kappa_{2}+1/2,\kappa_{1}-1/2\right)  }\left(  \cos2\theta\right)  .
\end{align*}
The norms with respect to $L^{2}\left(  \left[  -\pi,\pi\right]  ,\left\vert
\cos\theta\right\vert ^{2\kappa_{1}}\left\vert \sin\theta\right\vert
^{2\kappa_{2}}d\theta\right)  $ can be computed from%
\begin{align*}
&  \int_{-\pi}^{\pi}\left\{  P_{n}^{\left(  \alpha,\beta\right)  }\left(
\cos2\theta\right)  \right\}  ^{2}\left\vert \sin\theta\right\vert
^{2\alpha+1}\left\vert \cos\theta\right\vert ^{2\beta+1}d\theta\\
&  =\frac{4\Gamma\left(  \alpha+1\right)  \Gamma\left(  \beta+1\right)
}{\Gamma\left(  \alpha+\beta+2\right)  }\frac{\left(  \alpha+1\right)
_{n}\left(  \beta+1\right)  _{n}\left(  \alpha+\beta+n+1\right)  }{n!\left(
\alpha+\beta+2\right)  _{n}\left(  \alpha+\beta+2n+1\right)  },n\in
\mathbb{N}_{0}.
\end{align*}

\subsection{Dihedral groups}

The harmonic polynomials for the general dihedral groups can be expressed
using the two previous types and complex coordinates $z=x_{1}+\mathrm{i}%
x_{2},\overline{z}=x_{1}-\mathrm{i}x_{2}$. Interpret $p\left(  z^{m}\right)  $
as $p\left(  \operatorname{Re}z^{m},\operatorname{Im}z^{m}\right)  $. For the
one-parameter case $I_{2}\left(  m\right)  $ with $m$ odd the harmonic
polynomials are spanned by\newline1) $p_{n,0}\left(  z^{m}\right)
,p_{n,1}\left(  z^{m}\right)  $, of degree $nm$,\newline2) $\operatorname{Re}%
p_{nm+j}\left(  z\right)  ,\operatorname{Im}p_{nm+j}\left(  z\right)  $ where
$p_{nm+j}\left(  z\right)  :=z^{j}\left(  \frac{n+2\kappa}{2\kappa}%
p_{n,0}\left(  z^{m}\right)  +\mathrm{i}p_{n,1}\left(  z^{m}\right)  \right)
$, for $1\leq j<m.$

For the two-parameter case $I_{2}\left(  2m\right)  $ with $w_{\kappa}\left(
z\right)  =\left\vert z^{m}+\overline{z}^{m}\right\vert ^{2\kappa_{1}%
}\left\vert z^{m}-\overline{z}^{m}\right\vert ^{2\kappa_{2}}$ let%
\begin{align*}
q_{2n}\left(  z\right)   &  :=p_{2n,00}\left(  z\right)  +\frac{\mathrm{i}}%
{2}p_{2n,11}\left(  z\right)  ,\\
q_{2n+1}\left(  z\right)   &  :=\left(  n+\kappa_{2}+\frac{1}{2}\right)
p_{2n+1,10}\left(  z\right)  +i\left(  n+\kappa_{1}+\frac{1}{2}\right)
p_{2n+1,01}\left(  z\right)  .
\end{align*}
The harmonic polynomials are spanned by $z^{j}q_{n}\left(  z^{m}\right)  $ and
$\overline{z^{j}q_{n}\left(  z^{m}\right)  }$ for $n\in\mathbb{N}_{0},0\leq
j\leq m$ (note $q_{0}\left(  z\right)  =1$). These results are from \cite{Du1}.

\section{Nonsymmetric Jack polynomials}

For the symmetric group $\mathcal{S}_{d}$ acting naturally on $\mathbb{R}^{d}$
there is an elegant orthogonal basis for $\Pi^{d}$ with respect to the form
$\left\langle \cdot,\cdot\right\rangle _{\kappa}$. The basis consists of
nonsymmetric Jack polynomials, so named because their symmetrization (summing
over an $\mathcal{S}_{d}$-orbit) yields the Jack polynomials, with parameter
$1/\kappa$. The construction of the basis depends on commuting self-adjoint
\emph{\textquotedblleft Cherednik-Dunkl\textquotedblright\ operators} and an
ordering of the monomial basis with respect to which the operators are
represented by triangular matrices.

Recall the notation from Section \ref{c4rootA}: $\left(  i,j\right)  $ denotes
the transposition $s_{\varepsilon_{i}-\varepsilon_{j}}$ and $s_{i}=\left(
i,i+1\right)  $ for $1\leq i<d$. Interpret $\mathcal{S}_{d}$ as the set of
bijections on $\left\{  1,2,\ldots,d\right\}  $ then the action on
$\mathbb{R}^{d}$ is given by $\left(  w^{-1}x\right)  _{i}=x_{w\left(
i\right)  }$ and the action on monomials is $w\left(  x^{\alpha}\right)
=x^{\alpha w^{-1}}$ where $\left(  \alpha w^{-1}\right)  _{i}=\alpha
_{w^{-1}\left(  i\right)  }$ for $w\in\mathcal{S}_{d},1\leq i\leq d,\alpha
\in\mathbb{N}_{0}^{d}$.

\begin{definition}
The set of \emph{partitions} (of length $\leq d$) is
\[
\mathbb{N}_{0}^{d,+}:=\left\{  \lambda\in\mathbb{N}_{0}^{d}:\lambda_{i}%
\geq\lambda_{i+1},1\leq i<d\right\}  ,
\]
and for $\alpha\in\mathbb{N}_{0}^{d}$ let $\alpha^{+}$ denote the unique
partition such that $\alpha^{+}=\alpha w$ for some $w\in\mathcal{S}_{d}$.
\end{definition}

\begin{definition}
For $\alpha\in\mathbb{N}_{0}^{d}$ and $1\leq i\leq d$ let%
\[
w_{\alpha}\left(  i\right)  :=\#\left\{  j:\alpha_{j}>\alpha_{i}\right\}
+\#\left\{  j:1\leq j\leq i,\alpha_{j}=\alpha_{i}\right\}
\]
be the rank function.
\end{definition}

Note that $w_{\alpha}\left(  i\right)  <w_{\alpha}\left(  j\right)  $ is
equivalent to $\alpha_{i}>\alpha_{j}$, or $\alpha_{i}=\alpha_{j}$ and $i<j$.
For any $\alpha$ the function $w_{\alpha}$ is one-to-one on $\left\{
1,2,\ldots,d\right\}  $, hence $w_{\alpha}\in\mathcal{S}_{d}$. Also $\alpha$
is a partition if and only if $w_{\alpha}\left(  i\right)  =i$ for all $i$. In
general $\alpha w_{\alpha}^{-1}=\alpha^{+}$ because $\left(  \alpha w_{\alpha
}^{-1}\right)  _{i}=\alpha_{w_{\alpha}^{-1}\left(  i\right)  }$ for $1\leq
i\leq d$.

There is one conjugacy class of reflections and we use $\kappa$ for the value
of the multiplicity function. For $p\in\Pi^{d}$ and $1\leq i\leq d$%
\[
\mathcal{D}_{i}p\left(  x\right)  =\frac{\partial}{\partial x_{i}}p\left(
x\right)  +\kappa\sum_{j\neq i}\frac{p\left(  x\right)  -p\left(  \left(
i,j\right)  x\right)  }{x_{i}-x_{j}}.
\]
The commutation relations from Proposition \ref{c4Dmt} (using $x_{i}$ to
denote the multiplication operator) become%
\begin{align}
\left[  \mathcal{D}_{i},x_{i}\right]   &  =1+\kappa\sum_{j\neq i}\left(
i,j\right)  ,\label{DxixD}\\
\left[  \mathcal{D}_{j},x_{i}\right]   &  =-\kappa\left(  i,j\right)  ,~j\neq
i.\nonumber
\end{align}

The order on compositions is derived from the dominance order on partitions.

\begin{definition}
For $\alpha,\beta\in\mathbb{N}_{0}^{d}$ the partial order $\alpha\succ\beta$
($\alpha$ dominates $\beta$) means that $\alpha\neq\beta$ and $\sum_{i=1}%
^{j}\alpha_{i}\geq\sum_{i=1}^{j}\beta_{i}$ for $1\leq j\leq d$; and
$\alpha\vartriangleright\beta$ means that $\left\vert \alpha\right\vert
=\left\vert \beta\right\vert $ and either $\alpha^{+}\succ\beta^{+}$ or
$\alpha^{+}=\beta^{+}$ and $\alpha\succ\beta$.
\end{definition}

For example $\left(  5,1,4\right)  \vartriangleright\left(  1,5,4\right)
\vartriangleright\left(  4,3,3\right)  $, while $\left(  1,5,4\right)  $ and
$\left(  6,2,2\right)  $ are not comparable in $\vartriangleright$. The
following hold for $\alpha\in\mathbb{N}_{0}^{d}$: (1) $\alpha^{+}%
\trianglerighteq\alpha$, (2) if $\alpha_{i}>\alpha_{j}$ and $i<j$ then
$\alpha\vartriangleright\alpha\left(  i,j\right)  $, (3) if $1\leq
m<\alpha_{i}-\alpha_{j}$ then $\alpha^{+}\vartriangleright\left(
\alpha-m\left(  \varepsilon_{i}-\varepsilon_{j}\right)  \right)  ^{+}.$

The Cherednik-Dunkl operators are extensions of the Jucys-Murphy elements; for
$1\leq i\leq d$ let
\[
\mathcal{U}_{i}:=\mathcal{D}_{i}x_{i}-\kappa\sum_{1\leq j<i}\left(
i,j\right)  .
\]

\begin{proposition}
\label{c4propU}The operators $\mathcal{U}_{i}$ satisfy:\newline1) $\left[
\mathcal{U}_{i},\mathcal{U}_{j}\right]  =0$ for $1\leq i,j\leq d$,\newline2)
$\left\langle \mathcal{U}_{i}p,q\right\rangle _{\kappa}=\left\langle
p,\mathcal{U}_{i}q\right\rangle _{\kappa}$ for $p,q\in\Pi^{d}$,\newline3)
$s_{j}\mathcal{U}_{i}s_{j}=\mathcal{U}_{i}$ for $j\neq i-1,i$ and
$s_{i}\mathcal{U}_{i}s_{i}=\mathcal{U}_{i+1}+\kappa s_{i}$ for $1\leq i\leq d$.
\end{proposition}

\begin{proposition}
Let $\alpha\in\mathbb{N}_{0}^{d}$ and $1\leq i\leq d$, then%
\[
\mathcal{U}_{i}x^{\alpha}=\left(  \left(  d-w_{\alpha}\left(  i\right)
\right)  \kappa+\alpha_{i}+1\right)  x^{\alpha}+q_{\alpha,i}\left(  x\right)
,
\]
where $q_{\alpha,i}\left(  x\right)  $ is a sum of terms $\pm\kappa
x^{\alpha\left(  i,j\right)  }$ with $\alpha\vartriangleright\alpha\left(
i,j\right)  $ and $j\neq i$.
\end{proposition}

This shows that the matrix representing $\mathcal{U}_{i}$ on the monomial
basis of $\Pi_{n}^{d}$ for any $n\in\mathbb{N}_{0}$ is triangular (recall any
partial order can be embedded in a total order) and the eigenvalues of
$\mathcal{U}_{i}$ are $\left\{  \left(  d-w_{\alpha}\left(  i\right)  \right)
\kappa+\alpha_{i}+1:\alpha\in\mathbb{N}_{0}^{d},\left\vert \alpha\right\vert
=n\right\}  $.

For $\alpha\in\mathbb{N}_{0}^{d}$ and $1\leq i\leq d$ let%
\[
\xi_{i}\left(  \alpha\right)  :=\left(  d-w_{\alpha}\left(  i\right)  \right)
\kappa+\alpha_{i}+1.
\]
To assert that a commuting collection of triangular matrices have a basis of
joint eigenvectors a separation property suffices: if $\alpha,\beta
\in\mathbb{N}_{0}^{d}$ and $\alpha\neq\beta$ then $\xi_{i}\left(
\alpha\right)  \neq\xi_{i}\left(  \beta\right)  $ for some $i$. This condition
is satisfied if $\kappa$ is generic or $\kappa>0$ (in this case the
eigenvalues $\left\{  \xi_{i}\left(  \alpha\right)  :1\leq i\leq d\right\}  $
are pairwise distinct, and the largest value is $\left(  d-1\right)
\kappa+\alpha_{1}^{+}$; if $\xi_{i}\left(  \alpha\right)  =\xi_{i}\left(
\beta\right)  $ for all $i$ then $\alpha_{j}=\beta_{j}$ for $j=w_{\alpha}%
^{-1}\left(  1\right)  =w_{\beta}^{-1}\left(  1\right)  $, and so on).

\begin{theorem}
Suppose $\kappa$ is generic or $\kappa>0$, then for each $\alpha\in
\mathbb{N}_{0}^{N}$, there is a unique simultaneous eigenfunction
$\zeta_{\alpha}$ such that $\mathcal{U}_{i}\zeta_{\alpha}=\xi_{i}\left(
\alpha\right)  \zeta_{\alpha}$ for $1\leq i\leq d$ and
\[
\zeta_{\alpha}=x^{\alpha}+\sum\limits_{\alpha\vartriangleright\beta}%
A_{\beta\alpha}x^{\beta},
\]
with coefficients $A_{\beta\alpha}\in\mathbb{Q}\left(  \kappa\right)  $.
\end{theorem}

These eigenfunctions are called \emph{nonsymmetric Jack polynomials}, and form
a basis of $\Pi^{d}$ by the triangularity property. If $\alpha\neq\beta$ then
$\xi_{i}\left(  \alpha\right)  \neq\xi_{i}\left(  \beta\right)  $ for some $i$
and $\xi_{i}\left(  \alpha\right)  \left\langle \zeta_{\alpha},\zeta_{\beta
}\right\rangle _{\kappa}=\left\langle \mathcal{U}_{i}\zeta_{\alpha}%
,\zeta_{\beta}\right\rangle _{\kappa}=\left\langle \zeta_{\alpha}%
,\mathcal{U}_{i}\zeta_{\beta}\right\rangle _{\kappa}=\xi_{i}\left(
\beta\right)  \left\langle \zeta_{\alpha},\zeta_{\beta}\right\rangle _{\kappa
}$ and thus $\left\langle \zeta_{\alpha},\zeta_{\beta}\right\rangle _{\kappa
}=0$. The formula for $\left\langle \zeta_{\alpha},\zeta_{\alpha}\right\rangle
_{\kappa}$ is more complicated.

Suppose $\alpha\in\mathbb{N}_{0}^{d}$ and $\alpha_{i}<\alpha_{i+1}$ then
$\alpha s_{i}\vartriangleright\alpha$ and $w_{\alpha s_{i}}=w_{\alpha}s_{i}$;
let $p=s_{i}\zeta_{\alpha}-c\zeta_{\alpha}$ where $c\in\mathbb{Q}\left(
\kappa\right)  $ and is to be determined. By Proposition \ref{c4propU}
$\mathcal{U}_{j}p=\xi_{j}\left(  \alpha\right)  p$ for $j\neq i,i+1$ and the
leading term (with respect to $\vartriangleright$) in $p$ is $x^{\alpha s_{i}%
}$. Solve the equation $\mathcal{U}_{i}p=\xi_{i+1}\left(  \alpha\right)  p$
for $c$ by using $\mathcal{U}_{i}s_{i}=s_{i}\mathcal{U}_{i+1}+\kappa$ to
obtain $c=\frac{\kappa}{\xi_{i}\left(  \kappa\right)  -\xi_{i+1}\left(
\kappa\right)  }$. This implies that $\mathcal{U}_{i+1}p=\xi_{i}\left(
\alpha\right)  p,p=\zeta_{\alpha s_{i}}$ and
\begin{align*}
s_{i}\zeta_{\alpha}  &  =c\zeta_{\alpha}+\zeta_{\alpha s_{i}},\\
s_{i}\zeta_{\alpha s_{i}}  &  =\left(  1-c^{2}\right)  \zeta_{\alpha}%
-c\zeta_{\alpha s_{i}},\\
\left\langle \zeta_{\alpha s_{i}},\zeta_{\alpha s_{i}}\right\rangle _{\kappa}
&  =\left(  1-c^{2}\right)  \left\langle \zeta_{\alpha},\zeta_{\alpha
}\right\rangle _{\kappa}.
\end{align*}
The last equation follows from $\left\langle \zeta_{\alpha},\zeta_{\alpha
}\right\rangle _{\kappa}=\left\langle s_{i}\zeta_{\alpha},s_{i}\zeta_{\alpha
}\right\rangle _{\kappa}=c^{2}\left\langle \zeta_{\alpha},\zeta_{\alpha
}\right\rangle _{\kappa}+\left\langle \zeta_{\alpha s_{i}},\zeta_{\alpha
s_{i}}\right\rangle _{\kappa}$. The other ingredient is a raising operator.
From the commutations (\ref{DxixD}) we obtain:%
\begin{align*}
\mathcal{U}_{i}x_{d}  &  =x_{d}\left(  \mathcal{U}_{i}-\kappa\left(
i,d\right)  \right)  ,~1\leq i<d,\\
\mathcal{U}_{d}x_{d}  &  =x_{d}\left(  1+\mathcal{D}_{d}x_{d}\right)  .
\end{align*}
Let $\theta_{d}:=s_{1}s_{2}\ldots s_{d-1}$ thus $\theta_{d}\left(  d\right)
=1$ and $\theta_{d}\left(  i\right)  =i+1$ for $1\leq i<d$ (a cyclic shift).
Then%
\begin{align*}
\mathcal{U}_{i}x_{d}  &  =x_{d}\left(  \theta_{d}^{-1}\mathcal{U}_{i+1}%
\theta_{d}\right)  ,~1\leq i<d,\\
\mathcal{U}_{d}x_{d}  &  =x_{d}\left(  1+\theta_{d}^{-1}\mathcal{U}_{1}%
\theta_{d}\right)  .
\end{align*}
If $p$ satisfies $\mathcal{U}_{i}p=\lambda_{i}p$ for $1\leq i\leq d$ then
$\mathcal{U}_{i}\left(  x_{d}\theta_{d}^{-1}f\right)  =\lambda_{i+1}\left(
x_{d}\theta_{d}^{-1}f\right)  $ for $1\leq i<d$ and $\mathcal{U}_{d}\left(
x_{d}\theta_{d}^{-1}f\right)  =\left(  \lambda_{1}+1\right)  \left(
x_{d}\theta_{d}f\right)  $. For $\alpha\in\mathbb{N}_{0}^{d}$ let $\phi\left(
\alpha\right)  :=\left(  \alpha_{2},\alpha_{3},\ldots,\alpha_{d},\alpha
_{1}+1\right)  $, then $x_{d}\theta_{d}^{-1}x^{\alpha}=x^{\phi\left(
\alpha\right)  }$.

\begin{proposition}
If $\alpha\in\mathbb{N}_{0}^{d}$ then $\zeta_{\phi\left(  \alpha\right)
}=x_{d}\theta_{d}^{-1}\zeta_{\alpha}$ and
\[
\left\langle \zeta_{\phi\left(  \alpha\right)  },\zeta_{\phi\left(
\alpha\right)  }\right\rangle _{\kappa}\allowbreak=\left(  \left(
d-w_{\alpha}\left(  1\right)  \right)  \kappa+\alpha_{1}+1\right)
\left\langle \zeta_{\alpha},\zeta_{\alpha}\right\rangle _{\kappa}.
\]

\end{proposition}

\begin{proof}
The first part is shown by identifying the eigenvalues $\xi_{i}\left(
\phi\left(  \alpha\right)  \right)  $. Also $\left\langle \zeta_{\phi\left(
\alpha\right)  },\zeta_{\phi\left(  \alpha\right)  }\right\rangle _{\kappa
}=\left\langle \theta_{d}^{-1}\zeta_{\alpha},\mathcal{D}_{d}x_{d}\theta
_{d}^{-1}\zeta_{\alpha}\right\rangle _{\kappa}=\left\langle \theta_{d}%
^{-1}\zeta_{\alpha},\theta_{d}^{-1}\mathcal{D}_{1}x_{1}\zeta_{\alpha
}\right\rangle _{\kappa}=\xi_{1}\left(  \alpha\right)  \left\langle
\zeta_{\alpha},\zeta_{\alpha}\right\rangle _{\kappa}$.
\end{proof}

The norm formula involves a \emph{hook-length product}. For $\alpha
\in\mathbb{N}_{0}^{d}$ let $\ell\left(  \alpha\right)  =\max\left\{
j:\alpha_{j}>0\right\}  $, the length of $\alpha$. For a point $\left(
i,j\right)  $ in the \emph{Ferrers diagram} $\left\{  \left(  k,l\right)
\in\mathbb{N}_{0}^{2}:1\leq k\leq\ell\left(  \alpha\right)  ,1\leq l\leq
\alpha_{k}\right\}  $ and $t\in\mathbb{Q}\left(  \kappa\right)  $ let%
\begin{align*}
h\left(  \alpha,t;i,j\right)   &  :=\alpha_{i}-j+t+\kappa\#\left\{
l:l>i,j\leq\alpha_{l}\leq\alpha_{i}\right\} \\
&  +\kappa\#\left\{  l:l<i,j\leq\alpha_{l}+1\leq\alpha_{i}\right\}  ,
\end{align*}
and let%
\[
h\left(  \alpha,t\right)  :=\prod_{i=1}^{\ell\left(  \alpha\right)  }%
\prod_{j=1}^{\alpha_{i}}h\left(  \alpha,t;i,j\right)  .
\]
For $\lambda\in\mathbb{N}_{0}^{d,+}$ let $\left(  t\right)  _{\lambda}%
:=\prod_{i=1}^{d}\left(  t-(i-1)\kappa\right)  _{\lambda_{i}}$, the
\emph{generalized Pochhammer symbol}.

\begin{theorem}
For $\alpha\in\mathbb{N}_{0}^{d}$,%
\begin{align*}
\left\langle \zeta_{\alpha},\zeta_{\alpha}\right\rangle _{\kappa}  &  =\left(
d\kappa+1\right)  _{\alpha^{+}}\frac{h\left(  \alpha,1\right)  }{h\left(
\alpha,\kappa+1\right)  },\\
\zeta_{\alpha}\left(  1,\ldots,1\right)   &  =\frac{\left(  d\kappa+1\right)
_{\alpha^{+}}}{h\left(  \alpha,\kappa+1\right)  }.
\end{align*}

\end{theorem}

These formulae are proved by induction starting with $\alpha=0$ and using the
steps $\alpha\rightarrow\phi\left(  \alpha\right)  ,\alpha\rightarrow\alpha
s_{i}$ for $\alpha_{i}<\alpha_{i+1}$ (it suffices to use $\lambda
\rightarrow\phi\left(  \lambda\right)  $ with $\lambda\in\mathbb{N}_{0}^{d,+}$
in the computation). The number of such steps in any sequence linking $0$ to
$\alpha$ equals
\[
\left\vert \alpha\right\vert +\frac{1}{2}\sum\limits_{1\leq i<j\leq d}\left\{
\left\vert \alpha_{i}-\alpha_{j}\right\vert +\left\vert \alpha_{i}-\alpha
_{j}+1\right\vert -1\right\}  .
\]

There are two explicit results for $\mathcal{D}_{i}\zeta_{\alpha}$. Recall
$\theta_{m}=s_{1}s_{2}\ldots s_{m-1}$ for $m\leq d$.

\begin{proposition}
Suppose $\alpha\in\mathbb{N}_{0}^{d}$ and $\ell\left(  \alpha\right)  =m$; let
$\widetilde{\alpha}=\left(  \alpha_{m}-1,\alpha_{1},\ldots,\alpha
_{m-1},0,\ldots\right)  $ and $\beta_{m}=\alpha_{m}-w_{\alpha}\left(
m\right)  \kappa$ then%
\begin{align*}
\mathcal{D}_{i}\zeta_{\alpha}  &  =0,~m<i\leq d,\\
\mathcal{D}_{m}\zeta_{\alpha}  &  =\frac{\left(  m\kappa+\beta_{m}\right)
\left(  \left(  d+1\right)  \kappa+\beta_{m}\right)  }{\left(  m+1\right)
\kappa+\beta_{m}}\theta_{m}^{-1}\zeta_{\widetilde{\alpha}}.
\end{align*}

\end{proposition}

This result is from \cite[Proposition 3.17]{Du4}.

Symmetric polynomials (that is, $\mathcal{S}_{d}$-invariant) have bases
labeled by partitions $\lambda\in\mathbb{N}_{0}^{d,+}$. We describe bases
whose elements are mutually orthogonal in $\left\langle \cdot,\cdot
\right\rangle _{\kappa}$. For a given $\lambda\in\mathbb{N}_{0}^{d,+}$ one can
consider $\sum_{w\in\mathcal{S}_{d}}w\zeta_{\lambda}$ or a sum $\sum
_{\alpha^{+}=\lambda}b_{\alpha}\zeta_{\alpha}$ with suitable coefficients
$\left\{  b_{\alpha}\right\}  $. Let $\lambda^{R}=\left(  \lambda_{d}%
,\lambda_{d-1},\ldots,\lambda_{1}\right)  $, thus $\lambda^{R}$ is the unique
$\vartriangleright$-minimum in $\left\{  \alpha:\alpha^{+}=\lambda\right\}  $.
No term $x^{\alpha}$ with $\alpha^{+}=\lambda$ except $\alpha=\lambda^{R}$
appears in $\zeta_{\lambda^{R}}$. Also let $n_{\lambda}=\#\left\{
w\in\mathcal{S}_{d}:\lambda w=\lambda\right\}  $.

\begin{definition}
For $\lambda\in\mathbb{N}_{0}^{d,+}$ let $j_{\lambda}:=h\left(  \lambda
,1\right)  \sum\limits_{\alpha^{+}=\lambda}\dfrac{\zeta_{\alpha}}{h\left(
\alpha,1\right)  }$.
\end{definition}

\begin{theorem}
For $\lambda\in\mathbb{N}_{0}^{d,+}$ $j_{\lambda}$ has the following
properties:\newline1) $j_{\lambda}$ is symmetric, and the coefficient of
$x^{\lambda}$ in $j_{\lambda}$ is $1$,\newline2) $j_{\lambda}=\dfrac{h\left(
\lambda,\kappa+1\right)  }{n_{\lambda}h\left(  \lambda^{R},\kappa+1\right)
}\sum\limits_{w\in\mathcal{S}_{d}}w\zeta_{\lambda}=\dfrac{1}{n_{\lambda}}%
\sum\limits_{w\in\mathcal{S}_{d}}w\zeta_{\lambda^{R}}$,\newline3)
$\left\langle j_{\lambda},j_{\lambda}\right\rangle _{\kappa}=\left(
d\kappa+1\right)  _{\lambda}\dfrac{d!h\left(  \lambda,1\right)  }{n_{\lambda
}h\left(  \lambda^{R},\kappa+1\right)  }$ and $j_{\lambda}\left(
1,1,\ldots,1\right)  =\dfrac{d!\left(  d\kappa+1\right)  _{\lambda}%
}{n_{\lambda}h\left(  \lambda^{R},\kappa+1\right)  }$.
\end{theorem}

Knop and Sahi \cite{KS} proved the norm formulas and also showed that the
coefficient of each monomial in $h\left(  \alpha,\kappa+1\right)
\zeta_{\alpha}$ is a polynomial in $\kappa$ with nonnegative coefficients.

For $\lambda\in\mathbb{N}_{0}^{d,+}$ $j_{\lambda}$ is a scalar multiple of the
\emph{Jack polynomial} $J_{\lambda}\left(  x;\frac{1}{\kappa}\right)  $. For
more details on the nonsymmetric Jack polynomials and their applications to
Calogero-Moser-Sutherland systems and the groups of type $B$ see the monograph
\cite[Chapters 8,9]{DX}.

\end{document}